\documentclass{amsart}
\usepackage{amsmath,amsthm, amscd, amssymb, amsfonts}
\usepackage{mathrsfs}
\usepackage[english]{babel}
\usepackage[utf8]{inputenc}
\usepackage{graphicx}
\usepackage{enumerate}
\usepackage{esint}
\usepackage{color}
\usepackage{hyperref}
\usepackage{cleveref}
\usepackage[shortlabels]{enumitem}

\newtheorem{thm}{Theorem}
\newtheorem{cor}{Corollary}
\newtheorem{lem}{Lemma}
\newtheorem{prop}{Proposition}

\theoremstyle{remark}
\newtheorem{rem}{Remark}

\newcommand{\de}{\delta}

\newcommand{\ga}{\gamma}

\newcommand{\RR}{\mathbb{R}}

\newcommand{\NN}{\mathbb{N}}

\definecolor{celestialblue}{rgb}{0.29, 0.59, 0.82}
\definecolor{chocolate(web)}{rgb}{0.82, 0.41, 0.12}

\hypersetup{
	colorlinks = true,
	linkcolor = celestialblue,
	anchorcolor = blue,
	citecolor = magenta,
	filecolor = blue,
	urlcolor = chocolate(web)
}

\title[Mixed inequalities for operators associated to critical radius functions...]{Mixed inequalities for operators associated to critical radius functions with applications to Schr\"odinger type operators}
\author[F. Berra]{Fabio Berra}
\address{CONICET and Departamento de Matem\'{a}tica (FIQ-UNL),  Santa Fe, Argentina.}
\email{fberra@santafe-conicet.gov.ar}

\author[G. Pradolini]{Gladis Pradolini}
\address{CONICET and Departamento de Matem\'{a}tica (FIQ-UNL),  Santa Fe, Argentina.}
\email{gladis.pradolini@gmail.com}
\author[P. Quijano]{Pablo Quijano}
\address{Instituto de Matem\'atica Aplicada del Litoral, CONICET-UNL, and Facultad de Ingenier\'ia Qu\'imica, UNL. \indent Colectora Ruta Nac. N 168, Paraje El Pozo.
3000 Santa Fe, Argentina.}
\email{pquijano@santafe-conicet.gov.ar}

\thanks{The author were supported by CONICET and UNL.}



\date{}
\newcounter{BPQ}

\begin{document}

\begin{abstract}
We obtain weighted mixed inequalities for operators associated to a critical radius function. We consider 
Schrödinger Calderón-Zygmund operators of $(s,\delta)$ type, for $1<s\leq \infty$ and $0<\delta \leq 1$. We also give estimates of the same type for the associated maximal operators. As an application, we obtain a wide variety of mixed inequalities for Schrödinger type singular integrals.

As far as we know, these results are a first approach of mixed inequalities in the Schrödinger setting.
\end{abstract}

\maketitle

\section{Introduction}\label{seccion: introduccion}

One of the most classical results in Harmonic Analysis is the characterization of all the measurable and nonnegative functions $w$ for which the Hardy-Littlewood maximal operator $M$ maps $L^p(w)$ into $L^p(w)$, for $1<p<\infty$. B. Muckenhoupt solved this problem in \cite{Muck72}, showing that $M$ is bounded in $L^p(w)$ if and only if $w$ belongs to the $A_p$ class. Later on, this result was extended to the general context of spaces of homogeneous type (see, for example, \cite{Cald76} and \cite{MS81}).

In 1985, E. Sawyer proved in \cite{Sawyer} that if $u$ and $v$ are $A_1$ weights then the inequality
\begin{equation}\label{eq: intro - mixta de Sawyer}
uv\left(\left\{x\in \mathbb{R}: \frac{M(fv)(x)}{v(x)}>t\right\}\right)\leq \frac{C}{t}\int_{\mathbb{R}} |f|uv
\end{equation}
holds for every positive $t$. This result can be seen as the weak $(1,1)$ type of the operator $Sf=M(fv)/v$ with respect to the measure $d\mu(x)=u(x)v(x)\,dx$. The proof of this inequality is highly non-trivial and involves a very subtle decomposition of dyadic intervals into a family with certain properties, called ``principal intervals''. The idea is an adaptation of a technique that appears in \cite{M-W-76}. It is immediate that \eqref{eq: intro - mixta de Sawyer} generalizes the well-known fact that $M: L^{1}(u)\to L^{1,\infty}(u)$ when $u\in A_1$. On the other hand, one of the motivations to establish and prove \eqref{eq: intro - mixta de Sawyer} is that it allows to show, in an alternative way, the fact that $M$ is bounded in $L^p(w)$ when $w\in A_p$, by combining Jones' factorization theorem with Marcinkiewicz's interpolation result. The main difficulty that appears in proving the inequality above is that the classical covering lemmas do not apply for the operator $S$, which is a perturbation of $M$ via an $A_1$ weight. Moreover, the product $uv$ may be very singular. Indeed, if we take $u=v=|x|^{-1/2}$ then $uv$ is not even locally integrable. These reasons do not allow to apply classical techniques to solve the problem.

Later, in \cite{CruzUribe-Martell-Perez}, D. Cruz Uribe, J. M. Martell and C. P\'erez proved some extensions of inequality \eqref{eq: intro - mixta de Sawyer} to higher dimensions and  also for other operators. Particularly, they proved that if $u$ and $v$ are weights that satisfy  $u\in A_1$ and $v\in A_\infty(u)$, then the inequality
\begin{equation}\label{eq: intro - mixta de CUMP}
uv\left(\left\{x\in \mathbb{R}^d: \frac{|\mathcal{T}(fv)(x)|}{v(x)}>t\right\}\right)\leq \frac{C}{t}\int_{\mathbb{R}^d} |f|uv
\end{equation} 
holds for every positive $t$, where $\mathcal{T}$ is either the Hardy-Littlewood maximal function or a Calder\'on-Zygmund operator (CZO). The main idea in that paper is to obtain the desired estimate for $M_{\mathcal{D}}$, the dyadic Hardy-Littlewood maximal operator and then, by using an extrapolation result, derive the corresponding estimate for $M$ and CZOs. The condition on the weights guarantees that the product $uv$ is an $A_\infty$ weight and therefore some classical techniques, such as Calder\'on-Zygmund decomposition, can be applied to achieve the estimates.


We will refer to these type of estimates as mixed inequalities. Similar results for more general operators were also studied in the literature (see for example \cite{Berra-Carena-Pradolini(M)} and \cite{BCP21} for commutators of CZO, \cite{Berra-Carena-Pradolini(J)} for fractional operators, \cite{Berra} for generalized maximal functions and \cite{BCP21F-S} for mixed inequalities involving Fefferman-Stein estimates). 



In this paper we establish and prove mixed inequalities for classes of operators and weights associated to a \textit{critical radius function}. More precisely,  we will consider the space $\mathbb{R}^d$ equipped with a function $\rho:\mathbb{R}^d\rightarrow (0,\infty)$ whose variation is controlled by the existence of $C_0$ and $N_0\geq 1$ such that for every $x,y\in\mathbb{R}^d$
	\begin{equation} \label{eq-constantesRho}
	C_0^{-1}\rho(x) \left(1+ \frac{|x-y|}{\rho(x)}\right)^{-N_0}
	\leq \rho(y)
	\leq C_0 \,\rho(x) \left(1+ \frac{|x-y|}{\rho(x)}\right)^{\tfrac{N_0}{N_0+1}}.
	\end{equation}
	This functions and its associated operators appear naturaly when dealing with a Schr\"odinger operator $L=-\Delta + V$ as we shall discuss in \S~\ref{seccion: aplicaciones}. It is worth noting that if  $\rho$ is a critical radius function, then for any $\gamma>0$ the mapping $\gamma\rho$ is also a critical radius function. Moreover, if $0<\ga\leq 1$ then $\gamma\rho$ satisfies~\eqref{eq-constantesRho} with the same constants as $\rho$.

Given a locally integrable function $f$ and $\sigma\geq 0$, the \textit{Hardy-Littlewood maximal operator} $M^{\rho,\sigma}f$ is defined by
\begin{equation}\label{eq: operador maximal de H-L}
M^{\rho,\sigma}f(x)=\sup_{Q(x_0,r_0)\ni x} \left(1+\frac{r_0}{\rho(x_0)}\right)^{-\sigma}\left(\frac{1}{|Q|}\int_Q |f(y)|\,dy\right),
\end{equation}
where $Q(x_0,r_0)$ stands for the cube with sides parallel to the coordinate axes centered at $x_0$ and radius $r_0$, that is, $r_0=\sqrt{d}\,\ell(Q)/2$.
 Notice that $M^{\rho,0}=M$, the classical Hardy-Littlewood maximal function. This family of maximal operators is an adapted version of the Hardy-Littlewood maximal operator to the Schr\"odinger context and it is connected to the corresponding Muckenhoupt classes of weights $A^\rho_p$ associated to a critical radius function (see Proposition~3 in~\cite{BCHextrapolation}). 


Our first result involves a mixed type inequality for $M^{\rho,\theta}$. The weights involved in the estimate are a generalization of the classical Muckenhoupt classes and are associated to a critical radius function $\rho$ (see \S~\ref{seccion: preliminares} for the definition).

\begin{thm}\label{thm: mixta para M}
	Let $u\in A_1^{\rho}$ and $v\in A_\infty^\rho(u)$. Then there exists $\sigma\geq0$ such that the inequality 
	\begin{equation*}
	uv\left(\left\{x\in\mathbb{R}^d: \frac{M^{\rho,\sigma}(fv)(x)}{v(x)}>t\right\}\right)\leq \frac{C}{t}\int f(x)u(x)v(x)\,dx
	\end{equation*}
	holds for every positive $t$ and every bounded function with compact support.
\end{thm}


We shall also be dealing with singular integral operators related to a critical radius function $\rho$. 
For $0<\de\leq1$ we shall say that a linear operator $T$ is a  \emph{Schr\"odinger-Calder\'on-Zygmund operator (SCZO) of $(\infty,\delta)$} type if
\begin{enumerate}[\rm(I)]
	\item $T$ is bounded from $L^1$ into $L^{1,\infty}$;
	\item $T$ has an associated kernel $K:\mathbb{R}^d\times\mathbb{R}^d\rightarrow\mathbb{R}$, in the sense that
	\begin{equation*}
	Tf(x)=\int_{\mathbb{R}^d} K(x,y)f(y)\,dy,\,\,\,\,
	f\in L_c^{\infty} \,\,\text{and a.e.}\,\,x\notin \text{supp}f;
	\end{equation*}
	\item
	for each $N>0$ there exists a constant $C_N$ such that
	\begin{equation}\label{TamPuntual}
	|K(x,y)| \leq
	\frac{C_N}{|x-y|^{d}} \left(1+ \frac{|x-y|}{\rho(x)}\right)^{-N},\,\,\, x\neq y, 
	\end{equation}		
	and there exists $C$ such that
	\begin{equation}\label{suav-puntual}
	|K(x,y)-K(x,y_0)|
	\leq C \frac{|y-y_0|^{\de}}{|x-y|^{d+\de}},\,\,\,\text{when}\,\,
	|x-y|>2|y-y_0|.
	\end{equation}
\end{enumerate}

\begin{rem}\label{rem: tipo fuerte SCZ infinito}
    It is known that if $T$ is a SCZO of $(\infty,\delta)$ type then it is bounded on $L^p(w)$ for $1<p<\infty$ as long as $w\in A^\rho_p$ (see \S~\ref{seccion: preliminares} for the definition) and it is of weak type $(1,1)$ with respect to $w$ as long as $w\in A^\rho_1$ (see~\cite[Theorem 1]{BHQ1}, \cite[Theorem 5 and Proposition 5]{BCHextrapolation} and~\cite[Theorem 3.6]{BCH3}).
\end{rem}
Our main result dealing with this type of operators is contained in the following theorem.

\begin{thm}\label{thm: mixta para T SCZO (infinito,delta)}
	Let $\rho$ be a critical radius function, $u\in A_1^\rho$ and $v\in A_\infty^\rho(u)$. If $0<\delta\leq 1$ and $T$ is a SCZO of $(\infty,\delta)$ type, then the inequality
	\[uv\left(\left\{x\in\mathbb{R}^d: \frac{|T(fv)(x)|}{v(x)}>t\right\}\right)\leq \frac{C}{t}\int_{\mathbb{R}^d}|f(x)|u(x)v(x)dx\]
	holds for every positive $t$ and every bounded function with compact support.
\end{thm}

We shall also consider a wider class of operators with kernels satisfying another type of regularity. For $1<s<\infty$ and $0<\de\leq1$, we shall say that a linear operator $T$ is a  \emph{Schr\"odinger-Calder\'on-Zygmund operator (SCZO) of $(s,\delta)$ type} if
\begin{enumerate}
	\item[($\rm{I}_s$)]\label{tipodebil} T is bounded on $L^{p}$ for $1<p<s$. 
	\item[($\rm{II}_s$)] $T$ has an associated kernel $K:\mathbb{R}^d\times\mathbb{R}^d\rightarrow\mathbb{R}$, in the sense that
	\begin{equation*}
	Tf(x)=\int_{\mathbb{R}^d} K(x,y)f(y)\,dy,\,\,\,\,
	f\in L_c^{s'} \,\,\text{and}\,\,x\notin \text{supp}f.
	\end{equation*}
	Further, 
	for each $N>0$ there exists a constant $C_N$ such that
	\begin{equation}\label{TamHorm}
	\left( \int_{R<|x_0-x|<2R} |K(x,y)|^{s}\,dx \right)^{1/s}\\ \leq C_N R^{-d/s'} \left(1+\frac{R}{\rho(x_0)}\right)^{-N},
	\end{equation}
	for  $|y-x_0|<R/2$, and
	there exists $C$ such that
	\begin{equation}\label{suav-horm}
	\left( \int_{R<|x-y_0|<2R} |K(x,y)-K(x,y_0)|^{s}\,dx\right)^{1/s} \\
	\leq C R^{-d/s'} \left(\frac{r}{R}\right)^{\delta},
	\end{equation}
	for $|y-y_0|<r\leq \rho (y_0)$,  $r<R/2$.
\end{enumerate}
\begin{rem}\label{rem: tipo fuerte SCZ s}
    It is known that if $T$ is a SCZO of $(s,\delta)$ type for some $s>1$, then it is bounded on $L^p(w)$ for $1<p<s$ as long as $w^{1-p'}\in A^\rho_{p'/s'}$ and it is of weak type $(1,1)$ with respect to $w$ as long as $w^{s'}\in A^\rho_1$ (see~\cite[Theorem 1]{BHQ1}, \cite[Theorem 6 and Proposition 6]{BCHextrapolation} and~\cite[Theorem 3.6]{BCH3}).
\end{rem}


Our result for these type of operators is the following.

\begin{thm}\label{thm: mixta para T SCZO (s,delta)}
	Let $\rho$ be a critical radius function, $1<s<\infty$ and $0<\delta\leq 1$. Let $T$ be
 a SCZO of $(s,\delta)$ type. If $u$ is a weight verifying $u^{s'}\in A_1^{\rho}$ and $v\in A_\infty^{\rho}(u^\beta)$ for some $\beta>s'$, then the inequality
	\[uv\left(\left\{x\in\mathbb{R}^d: \frac{|T(fv)(x)|}{v(x)}>t\right\}\right)\leq \frac{C}{t}\int_{\mathbb{R}^d}|f(x)|u(x)v(x)dx\]
	holds for every positive $t$. 
\end{thm}

\begin{rem}
	Notice that we require a bit stronger assumption on $v$ for this theorem to hold. It is easy to see that the hypothesis on $u$ and $v$ in this theorem imply the corresponding ones in Theorem~\ref{thm: mixta para T SCZO (infinito,delta)} by virtue of Proposition~\ref{propo: v en A1(u) implica v en A1(u^alpha)} (see \S~\ref{seccion: preliminares}).
\end{rem}

Theorem~\ref{thm: mixta para T SCZO (infinito,delta)} and Theorem~\ref{thm: mixta para T SCZO (s,delta)} extend the known weak $(1,1)$ type of $T$ with $A_1^\rho$ weights when we take $v=1$ (see Theorem 3.6 in~\cite{BCH3}).

The article is organized as follows. In \S~\ref{seccion: preliminares} we give the  definitions for the reading and we state and prove some auxiliary results required for the main proofs, contained in \S~\ref{seccion: prueba de resultados principales}. Finally in \S~\ref{seccion: aplicaciones} we give mixed inequalities for operators in the Schrödinger setting, as an application. 

\section{Preliminaries and auxiliary results}\label{seccion: preliminares}

We begin by introducing the basic definitions involved in our estimates. Throughout the article $\rho: \mathbb{R}^d\to (0,\infty)$ will denote a fixed critical radius function. By a \textit{weight} $w$ we understand a function that is locally integrable and verifies $0<w(x)<\infty$ almost everywhere. 

We shall now introduce the classes of weights involved in our estimates. These are an extension of the classical Muckenhoupt $A_p$ classes and were first defined by B. Bongioanni, E. Harboure and O. Salinas in~\cite{BHS-classesofweights}.

Let $u$ be a weight. For $1<p<\infty$ and $\theta\geq 0$, we say that $w\in A_p^{\rho,\theta}(u)$ if the inequality
\begin{equation}\label{eq: clase Ap,rho,theta(u)}
\left(\frac{1}{u(Q)}\int_Q wu\right)^{1/p}\left(\frac{1}{u(Q)}\int_Q w^{1-p'}u\right)^{1/p'}\leq C\left(1+\frac{r}{\rho(x)}\right)^{\theta}
\end{equation} 
holds for every cube $Q=Q(x,r)$ and $C$ independent of $Q$. We recall that we consider cubes with sides parallel to the coordinate axes, $\ell(Q)$ denotes the length of the sides of $Q$ and $Q(x,r)$ denotes the cube centered in $x$ with radious $r$, that is, $r=\sqrt{d}\ell(Q)/2$. This notation will be used throughout the article. 

Similarly, $w\in A_1^{\rho,\theta}(u)$ if
\begin{equation}\label{eq: clase A1,rho,theta(u)}
\frac{1}{u(Q)}\int_Q wu\leq C\left(1+\frac{r}{\rho(x)}\right)^{\theta}\inf_Q w,
\end{equation} 
for every cube $Q$. The smallest constants in the corresponding inequalities above will be denoted by $[w]_{A_p^{\rho,\theta}(u)}$. For $1\leq p<\infty$, the $A_p^\rho$ class is defined as the collection of all the $A_p^{\rho,\theta}$ classes for $\theta\geq 0$, that is
\[A_p^\rho(u)=\bigcup_{\theta\geq 0} A_p^{\rho,\theta}(u).\]
We also define
\[A_\infty^\rho(u) =\bigcup_{p\geq 1} A_p^\rho(u).\]

There are many ways to characterize the class above. A weight $w$ belongs to $A_\infty^\rho(u)$ if there exist $\theta\geq 0$ and $\varepsilon>0$ such that the inequality
\begin{equation}\label{eq: condicion Ainf,rho(u)}
\frac{wu(E)}{wu(Q)}\leq C\left(1+\frac{r}{\rho(x)}\right)^\theta\left(\frac{u(E)}{u(Q)}\right)^\varepsilon
\end{equation}
holds for every cube  $Q$ and every measurable subset $E$ of $Q$.

When $u=1$ we shall simply denote $A_p^\rho(u)=A_p^\rho$ and $A_p^{\rho,\theta}(u)=A_p^{\rho,\theta}$, $1\leq p\leq \infty$ and $\theta\geq 0$.

A very important property of $A_p^\rho$ weights is that they verify a reverse H\"{o}lder inequality. More precisely, given $\theta\geq 0$ and $1<s<\infty$, we say that a weight $w$ belongs to the \textit{reverse H\"{o}lder class} $\mathrm{RH}_s^{\rho,\theta}$ if there exists a positive constant $C$ such that for every cube $Q$  the inequality
\begin{equation}\label{eq: clase RHs,rho,theta}
\left(\frac{1}{|Q|}\int_Q w^s\right)^{1/s}\leq C\left(1+\frac{r}{\rho(x)}\right)^\theta\left(\frac{1}{|Q|}\int_Q w\right).
\end{equation}
holds. When $s=\infty$, we say $w\in \mathrm{RH}_{\infty}^{\rho,\theta}$ if
\begin{equation}\label{eq: clase RHinf,rho,theta}
\sup_Q w\leq C\left(1+\frac{r}{\rho(x)}\right)^\theta\left(\frac{1}{|Q|}\int_Q w\right)
\end{equation}
holds for every ball $Q$ and $C$ independent of $Q$. The smallest constant $C$ for which these estimates hold will be denoted by $[w]_{\mathrm{RH}_s^{\rho,\theta}}$.

As in the case of $A_p^\rho$ classes, we define
\[\mathrm{RH}_s^\rho=\bigcup_{\theta\geq 0}\mathrm{RH}_s^{\rho,\theta}, \quad 1<s\leq\infty.\] 
The next results give some useful properties of weights belonging to $A_p^\rho$ and $\mathrm{RH}_s^\rho$ classes, generalizing some classical versions which deal with usual $A_p$ and $\mathrm{RH}_s$ classes.

In order to do so, we shall introduce an auxiliary operator to establish some properties of weights that will be useful in our main results. For $\theta\geq 0$ we define the \textit{Hardy-Littlewood minimal operator}  to be
\begin{equation}\label{eq: operador minimal}
\mathcal{M}^{\rho,\theta}f(x)=\inf_{Q(x_0,r_0)\ni x} \left(1+\frac{r_0}{\rho(x_0)}\right)^{\theta}\left(\frac{1}{|Q|}\int_Q |f(y)|\,dy\right)
\end{equation}
The operators $M^{\rho,\theta}$ and $\mathcal{M}^{\rho,\theta}$ are closely related to weights belonging to $A_1^{\rho,\theta}$ and $\mathrm{RH}_\infty^{\rho,\theta}$, respectively. This relation is established and proved in Lema~\ref{lema: relacion w-Mw para maximal y minimal} (see \S~\ref{seccion: preliminares}).

The next lemma establishes a characterization for weights belonging to $A_1^{\rho,\theta}$ and $\textrm{RH}_\infty^{\rho,\theta}$. The proof is straightforward so we will omit it.

\begin{lem}\label{lema: relacion w-Mw para maximal y minimal}
	Let $w$ be a weight and $\theta\geq 0$. Then:
	\begin{enumerate}[\rm(a)]
		\item\label{item: lema: relacion w-Mw para maximal y minimal - item a} $w\in A_1^{\rho,\theta}$ if and only if there exists a positive constant $C$ such that $M^{\rho,\theta}w(x)\leq Cw(x)$, for almost every $x$;
		\item\label{item: lema: relacion w-Mw para maximal y minimal - item b} $w\in \mathrm{RH}_\infty^{\rho,\theta}$ if and only if there exists $C>0$ such that the inequality $w(x)\leq~C\mathcal{M}^{\rho,\theta}w(x)$ holds for almost every $x$.
	\end{enumerate}
\end{lem}

The lemma below states some properties of $A_p^\rho$ weights. These properties are well-known in the literature. A proof can be found in~\cite{BHS-classesofweights} and~\cite{BHQ-twoweighted}.

\begin{lem}\label{lema: propiedades clase Ap,rho}
	Let $\theta\geq 0$ and $1<p<\infty$. Then:
	\begin{enumerate}[\rm(a)]
		\item \label{item: lema: propiedades clase Ap,rho - item a}if $u\in A_p^\rho$, then $u^{1-p'}\in A_{p'}^\rho$;
		\item \label{item: lema: propiedades clase Ap,rho - item b}if $u$ and $v$ belong to $A_1^\rho$, then $uv^{1-p}\in A_p^\rho$;
		\item \label{item: lema: propiedades clase Ap,rho - item c}if $w\in A_p^\rho$, there exist weights $u$ and $v$ in $A_1^\rho$ that verify $w=uv^{1-p}$.
	\end{enumerate}
\end{lem}

The following proposition gives us another factorization property of $A_p^\rho$ weights. The corresponding result for classical Muckenhoupt and reverse Hölder classes was proved in \cite{Cruz-Uribe-Neugebauer}.

\begin{prop}\label{propo: pesos de Ainf,rho se factorizan}
	Let $s>1$, $1<p<\infty$ and $w\in A_p^\rho\cap \mathrm{RH}_s^\rho$. Then there exist weights $w_1$ and $w_2$ such that $w=w_1w_2$, with $w_1\in A_1^\rho\cap \mathrm{RH}_s^\rho$ and $w_2\in A_p^\rho\cap \mathrm{RH}_\infty^\rho$.
\end{prop}

The proof of this result will follow from the next two lemmas, which are extensions of classical properties given also in \cite{Cruz-Uribe-Neugebauer}.

\begin{lem}\label{lema: relacion Ap,rho-RHs-rho con Aq-rho}
	Let $s,p>1$ and $q$ defined by $q=s(p-1)+1$. Then $w\in A_p^\rho\cap\mathrm{RH}_s^\rho$ if and only if $w^s\in A_q^\rho$.
\end{lem}

\begin{proof}
	Let us first assume that $w\in A_p^\rho\cap \mathrm{RH}_s^\rho$. Then there exist two nonnegative numbers $\theta_1$ and $\theta_2$ such that $w\in \mathrm{RH}_s^{\rho,\theta_1}\cap A_p^{\rho,\theta_2}$. By using the relation between $p$, $q$ and $s$ we have, for every cube $Q=Q(x,r)$
	\begin{align*}
	\left(\frac{1}{|Q|}\int_Q w^s\right)^{1/q}\left(\frac{1}{|Q|}\int_Q w^{s(1-q')}\right)^{1/q'}&=\left(\frac{1}{|Q|}\int_Q w^s\right)^{1/q}\left(\frac{1}{|Q|}\int_Q w^{1-p'}\right)^{s(p-1)/q}\\
	&\leq [w]_{\mathrm{RH}_s^{\rho,\theta_1}}^{s/q}\left(\frac{1}{|Q|}\int_Q w\right)^{s/q}\left(1+\frac{r}{\rho(x)}\right)^{\theta_1s/q}\\
	&\qquad\times\left(\frac{1}{|Q|}\int_Q w^{1-p'}\right)^{s(p-1)/q}\\
	&\leq [w]_{\mathrm{RH}_s^{\rho,\theta_1}}^{s/q}[w]_{A_p^{\rho,\theta_2}}^{ps/q}\left(1+\frac{r}{\rho(x)}\right)^{ps\theta_2/q+s\theta_1/q},
	\end{align*}
	which implies that $w^s\in A_q^{\rho,\theta_0}\subseteq A_q^\rho$, where $\theta_0=ps\theta_2/q+s\theta_1/q$.
	
	Conversely, assume that $w^s\in A_q^\rho$. Then there exists $\theta_1\geq 0$ such that $w^s\in A_q^{\rho,\theta_1}$. By Jensen's inequality we obtain that
	\begin{align*}
	\left(\frac{1}{|Q|}\int_Q w\right)^{1/p}\left(\frac{1}{|Q|}\int_Q w^{1-p'}\right)^{1/p'}&\leq \left(\frac{1}{|Q|}\int_Q w^s\right)^{1/(ps)}\left(\frac{1}{|Q|}\int_Q w^{(1-q')s}\right)^{(q-1)/(sp)}\\
	&\leq \left[\left(\frac{1}{|Q|}\int_Q w^s\right)^{1/q}\left(\frac{1}{|Q|}\int_Q w^{(1-q')s}\right)^{1/q'}\right]^{q/(sp)}\\
	&\leq [w^s]_{A_q^{\rho,\theta_1}}^{q/(sp)}\left(1+\frac{r}{\rho(x)}\right)^{q\theta_1/(sp)},
	\end{align*}
	that is, $w\in A_p^{\rho,q\theta_1/(sp)}\subseteq A_p^\rho$. In order to prove that $w\in \mathrm{RH}_s^\rho$, notice that
	\begin{align*}
	\frac{1}{|Q|}\int_Q w^s &= \left(\frac{1}{|Q|}\int_Q w^s\right)\left(\frac{1}{|Q|}\int_Q w^{s(1-q')}\right)^{q-1}\left(\frac{1}{|Q|}\int_Q w^{s(1-q')}\right)^{1-q}\\
	&\leq [w^s]_{A_q^{\rho,\theta_1}}^q\left(1+\frac{r}{\rho(x)}\right)^{\theta_1q}\left(\frac{1}{|Q|}\int_Q w^{1-p'}\right)^{(1-p)s}\\
	&\leq [w^s]_{A_q^{\rho,\theta_1}}^q\left(1+\frac{r}{\rho(x)}\right)^{\theta_1q}\left(\frac{1}{|Q|}\int_Q w\right)^s.
	\end{align*} 
	This implies that
	\[\left(\frac{1}{|Q|}\int_Q w^s\right)^{1/s}\leq [w^s]_{A_q^{\rho,\theta_1}}^{q/s}\left(1+\frac{r}{\rho(x)}\right)^{\theta_1q/s}\left(\frac{1}{|Q|}\int_Q w\right)\]
	and therefore $w\in\mathrm{RH}_s^{\rho,\theta_1q/s}\subseteq \mathrm{RH}_s^\rho$.
\end{proof}

\begin{rem}\label{obs: existen potencias mayores que 1 de A_1,rho en A1,rho}
	The previous result remains true when $p=1$. If $u$ is a weight in $A_1^\rho$, then there exists $s>1$ such that $u\in \mathrm{RH}_s^\rho$. Therefore, this lemma establishes that $u^s\in A_1^\rho$. In conclusion, if $u\in A_1^\rho$ then there exists $\varepsilon_0>0$ such that $u^{1+\varepsilon}\in A_1^\rho$, for every $0<\varepsilon\leq \varepsilon_0$. 
\end{rem}

\begin{lem}\label{lema: w en A1,rho implica w^{1-p} en RHinfty,rho}
	Let $w\in A_1^\rho$. Then $w^{1-p}\in A_p^\rho\cap \mathrm{RH}_\infty^\rho$, for every $1<p<\infty$. 
\end{lem}

\begin{proof}
	The fact that $w^{1-p}\in A_p^\rho$ follows from item~\ref{item: lema: propiedades clase Ap,rho - item b} in Lemma~\ref{lema: propiedades clase Ap,rho}. To prove that $w^{1-p}\in \mathrm{RH}_\infty^\rho$ we shall see that there exist $C>0$ and $\theta\geq 0$ such that $w^{1-p}(x)\leq C\mathcal{M}^{\rho,\theta}w^{1-p}(x)$, for almost every $x$. Then, the conclusion follows immediately from item~\ref{item: lema: relacion w-Mw para maximal y minimal - item b} of Lemma~\ref{lema: relacion w-Mw para maximal y minimal}. 
	
	Since $w\in A_1^\rho$, there exists $\theta\geq 0$ such that $w\in A_1^{\rho,\theta}$. For $x\in \mathbb{R}^d$ and every cube $Q=Q(x_0,r_0)$ that contains $x$ we have
	\begin{align*}
	1&\leq \left(\frac{1}{|Q|}\int_Q w\right)^{1/p'}\left(\frac{1}{|Q|}\int_Q w^{1-p}\right)^{1/p}\left(1+\frac{r_0}{\rho(x_0)}\right)^{-\theta/p'}\left(1+\frac{r_0}{\rho(x_0)}\right)^{\theta/p'}\\
	&\leq \left(M^{\rho,\theta}w(x)\right)^{1/p'}\left[\left(1+\frac{r_0}{\rho(x_0)}\right)^{\theta/(p'-1)}\left(\frac{1}{|Q|}\int_Q w^{1-p}\right)\right]^{1/p}.
	\end{align*}
	By taking the infimum over such cubes we obtain
	\[1\leq \left(M^{\rho,\theta}w(x)\right)^{1/p'}\left(\mathcal{M}^{\rho,\tilde\theta}w^{1-p}(x)\right)^{1/p},\]
	or equivalently
	\[1\leq \left(M^{\rho,\theta}w(x)\right)\left(\mathcal{M}^{\rho,\tilde\theta}w^{1-p}(x)\right)^{p'-1},\]
	where $\tilde \theta =\theta/(p'-1)$. This inequality combined with Lemma~\ref{lema: relacion w-Mw para maximal y minimal} allow us to conclude that there exists $C>0$ such that
	\[w^{-1/(p'-1)}(x)\leq C\mathcal{M}^{\rho,\tilde\theta}w^{1-p}(x),\]
	or
	\[w^{1-p}(x)\leq C\mathcal{M}^{\rho,\tilde\theta}w^{1-p}(x),\]
	which implies, again by item~\ref{item: lema: relacion w-Mw para maximal y minimal - item b} in Lemma~\ref{lema: relacion w-Mw para maximal y minimal}, that $w^{1-p}\in \mathrm{RH}_\infty^{\rho, \tilde \theta}\subseteq \mathrm{RH}_\infty^{\rho}$.
\end{proof}

As an immediate consequence of this lemma we have the following result.

\begin{cor}\label{coro: potencias negativas de A1,rho estan en RHinfty,rho}
	If $w\in A_1^\rho$ then $w^{-r}\in \mathrm{RH}_\infty^\rho$, for every $r>0$.
\end{cor}

We now proceed with the proof of Proposition~\ref{propo: pesos de Ainf,rho se factorizan}.

\begin{proof}[Proof of Proposition~\ref{propo: pesos de Ainf,rho se factorizan}]
	By virtue of Lemma~\ref{lema: relacion Ap,rho-RHs-rho con Aq-rho}, we have that $w^s\in A_q^\rho$. From item~\ref{item: lema: propiedades clase Ap,rho - item c} in Lemma~\ref{lema: propiedades clase Ap,rho}, there exist two weights $v_1$ and $v_2$ in $A_1^\rho$ such that $w^s=v_1v_2^{1-q}$. Then $w=v_1^{1/s}v_2^{(1-q)/s}$. We define $w_1=v_1^{1/s}$ and $w_2=v_2^{(1-q)/s}=v_2^{1-p}$.
	
	Observe that $w_1\in A_1^\rho$, since $v_1\in A_1^\rho$. Thus
	\begin{align*}
	\left(\frac{1}{|Q|}\int_Q w_1^s\right)^{1/s}&=\left(\frac{1}{|Q|}\int_Q v_1\right)^{1/s}\\
	&\leq \left([v_1]_{A_1^{\rho,\theta_1}} \left(1+\frac{r}{\rho(x)}\right)^{\theta_1}\inf_Q v_1\right)^{1/s}\\
	&\leq [v_1]_{A_1^{\rho,\theta_1}}^{1/s} \left(1+\frac{r}{\rho(x)}\right)^{\theta_1/s}\inf_Q v_1^{1/s}\\
	&\leq [v_1]_{A_1^{\rho,\theta_1}}^{1/s}\left(1+\frac{r}{\rho(x)}\right)^{\theta_1/s}\left(\frac{1}{|Q|}\int_Q w_1\right),
	\end{align*}
	for every cube $Q$. This establishes that $w_1\in \mathrm{RH}_s^{\rho,\theta_1/s}$ and consequently $w_1\in A_1^\rho\cap \mathrm{RH}_s^\rho$.
	
	On the other hand, Lemma~\ref{lema: w en A1,rho implica w^{1-p} en RHinfty,rho} yields $w_2\in A_p^\rho\cap \mathrm{RH}_\infty^\rho$. This completes the proof. \qedhere
\end{proof}

The following properties will be useful in the sequel.

\begin{lem}\label{lema: producto de promedios acotado por promedio del producto, p y p'}
	Let $1<p<\infty$. Let $u$ and $v$ two weights that verify $u\in A_p^\rho$ and $v\in A_{p'}^\rho$. Then, there exist $C>0$ and $\theta\geq0$ such that
	\[\left(\frac{1}{|Q|}\int_Q u\right)^{1/p}\left(\frac{1}{|Q|}\int_Q v\right)^{1/p'}\leq C\left(\frac{1}{|Q|}\int_Q u^{1/p}v^{1/p'}\right)\left(1+\frac{r}{\rho(x)}\right)^{\theta},\]
	for every cube  $Q$.
\end{lem}

\begin{proof}
	By hypothesis, there exist two nonnegative numbers $\theta_1$ and $\theta_2$ such that
	\[\left(\frac{1}{|Q|}\int_Q u\right)^{1/p}\left(\frac{1}{|Q|}\int_Q u^{1-p'}\right)^{1/p'}\leq [u]_{A_p^{\rho,\theta_1}}\left(1+\frac{r}{\rho(x)}\right)^{\theta_1}\]
	and
	\[\left(\frac{1}{|Q|}\int_Q v\right)^{1/p'}\left(\frac{1}{|Q|}\int_Q v^{1-p}\right)^{1/p}\leq [v]_{A_{p'}^{\rho,\theta_2}}\left(1+\frac{r}{\rho(x)}\right)^{\theta_2}.\]
	By multiplying these two inequalities, we get
	\begin{equation}\label{eq: lema: producto de promedios acotado por promedio del producto, p y p' - eq1}
	\left(\frac{1}{|Q|}\int_Q u\right)^{1/p}\left(\frac{1}{|Q|}\int_Q v\right)^{1/p'}\leq C \left(1+\frac{r}{\rho(x)}\right)^{\theta_1+\theta_2}\left(\frac{1}{|Q|}\int_Q u^{1-p'}\right)^{-1/p'}\left(\frac{1}{|Q|}\int_Q v^{1-p}\right)^{-1/p},
	\end{equation}
	where $C=[u]_{A_p^{\rho,\theta_1}}[v]_{A_{p'}^{\rho,\theta_2}}$. On the other hand, by H\"{o}lder inequality we have 
	\begin{equation}\label{eq: lema: producto de promedios acotado por promedio del producto, p y p' - eq2}
	\frac{1}{|Q|}\int_Q u^{-1/p}v^{-1/p'}\leq \left(\frac{1}{|Q|}\int_Q u^{1-p'}\right)^{1/p'}\left(\frac{1}{|Q|}\int_Q v^{1-p}\right)^{1/p}.
	\end{equation}
	By combining \eqref{eq: lema: producto de promedios acotado por promedio del producto, p y p' - eq1} and \eqref{eq: lema: producto de promedios acotado por promedio del producto, p y p' - eq2} with Jensen inequality we obtain
	\begin{align*}
	\left(\frac{1}{|Q|}\int_Q u\right)^{1/p}\left(\frac{1}{|Q|}\int_Q v\right)^{1/p'}&\leq C \left(1+\frac{r}{\rho(x)}\right)^{\theta_1+\theta_2}\left(\frac{1}{|Q|}\int_Q u^{-1/p}v^{-1/p'}\right)^{-1}\\
	&\leq C\left(1+\frac{r}{\rho(x)}\right)^{\theta_1+\theta_2}\frac{1}{|Q|}\int_Q u^{1/p}v^{1/p'}.\qedhere
	\end{align*}
\end{proof}

\begin{prop}\label{propo: relacion RHs con Ainf}
	Let $w$ be a weight. Then $w\in \mathrm{RH}_s^\rho$ if and only if $w^s\in A_\infty^\rho$.
\end{prop}

\begin{proof}
	Assume that $w\in \mathrm{RH}_s^\rho$. Then there exists $\theta_1\geq 0$ such that $w\in\mathrm{RH}_s^{\rho,\theta_1}$. If we prove that $\mathrm{RH}_s^\rho\subseteq A_\infty^\rho$, then there would exist $p_0>1$ such that $w\in A_{p_0}^\rho\cap\mathrm{RH}_s^\rho$. Lemma~\ref{lema: relacion Ap,rho-RHs-rho con Aq-rho} establishes that $w^s\in A_{q_0}^\rho$, where $q_0=s(p_0-1)+1$. This would imply that $w^s\in A_\infty^\rho$. Therefore, let us prove that $\mathrm{RH}_s^\rho\subseteq A_\infty^\rho$. Fix a cube $Q$ and $E$ a measurable subset of $Q$. We have
	\begin{align*}
	w(E)&\leq \left(\int_E w^s\right)^{1/s}|E|^{1/s'}\\
	&\leq \left(\frac{1}{|Q|}\int_Q w^s\right)^{1/s}|Q|^{1/s}|E|^{1/s'}\\
	&\leq [w]_{\mathrm{RH}_s^{\rho,\theta_1}}w(Q)\left(1+\frac{r}{\rho(x)}\right)^{\theta_1}\left(\frac{|E|}{|Q|}\right)^{1/s'},
	\end{align*}
	and this yields
	\[\frac{w(E)}{w(Q)}\leq [w]_{\mathrm{RH}_s^{\rho,\theta_1}}\left(1+\frac{r}{\rho(x)}\right)^{\theta_1}\left(\frac{|E|}{|Q|}\right)^{1/s'},\]
	that is, $w\in A_\infty^\rho$.
	
	Conversely, assume that $w^s\in A_\infty^\rho$. This implies that $w^s\in A_{q_0}^\rho$, for some $q_0>1$. By virtue of Lemma~\ref{lema: relacion Ap,rho-RHs-rho con Aq-rho}, we have that $w\in \mathrm{RH}_s^\rho\cap A_{p_0}^\rho$, where $p_0=1+(q_0-1)/s$.
\end{proof}

\medskip

\begin{lem}\label{lema: Holder al reves con producto de pesos}
	Let $p>1$, $u\in\mathrm{RH}_p^\rho$ and $v\in \mathrm{RH}_{p'}^\rho$. Then there exist $\theta\geq 0$ and a positive constant $C$ such that
	\[\left(\int_Q u^p\right)^{1/p}\left(\int_Q v^{p'}\right)^{1/p'}\leq C\left(1+\frac{r}{\rho(x)}\right)^{\theta}\left(\int_Q uv\right),\]
	for every cube $Q$.
\end{lem}

\begin{proof}
	Let us prove, as a first step, that there exists $0<\delta<1$ such that $u^{\delta p}\in A_p^\rho\cap\mathrm{RH}_{1/\delta}^{\rho}$ and $v^{\delta p'}\in A_{p'}^\rho\cap\mathrm{RH}_{1/\delta}^{\rho}$. Since $u\in \mathrm{RH}_p^\rho$ and $v\in  \mathrm{RH}_{p'}^\rho$, Proposition~\ref{propo: relacion RHs con Ainf} implies that there exist  $1<q,r<\infty$ verifying $u\in A_q^\rho$ and $v\in A_r^\rho$. Define $p_0=\min\{p,p'\}$, $q_0=\max\{q,r\}$ and $\delta=(p_0-1)/(q_0-1)$. Observe that $0<\delta<1$ because $q_0$ can be chosen arbitrarily large. Also notice that $u^p$ and $v^{p'}$ belong to the $A_{q_0}^\rho$ class, with $q_0=1+(p_0-1)/\delta$. Lemma~\ref{lema: relacion Ap,rho-RHs-rho con Aq-rho} yields $u^{\delta p}\in A_{p_0}^\rho\cap \mathrm{RH}_{1/\delta}^\rho$ and $v^{\delta p'}\in A_{p_0}^\rho\cap \mathrm{RH}_{1/\delta}^\rho$. Therefore
	\[u^{\delta p}\in A_{p}^\rho\cap \mathrm{RH}_{1/\delta}^\rho\quad\textrm{ and }\quad v^{\delta p'}\in A_{p'}^\rho\cap \mathrm{RH}_{1/\delta}^\rho.\] 
	By combining these facts with Lemma~\ref{lema: producto de promedios acotado por promedio del producto, p y p'} we obtain that
	\begin{align*}
	\left(\frac{1}{|Q|}\int_Q u^p\right)^{\delta/p}\left(\frac{1}{|Q|}\int_Q v^{p'}\right)^{\delta/p'}&\leq [u^{\delta p}]_{\mathrm{RH}_{1/\delta}^{\rho,\theta_1}}^{1/p}[v^{\delta p'}]_{\mathrm{RH}_{1/\delta}^{\rho,\theta_2}}^{1/p'}\left(1+\frac{r}{\rho(x)}\right)^{\theta_1+\theta_2}\\
	&\qquad\times \left(\frac{1}{|Q|}\int_Q u^{\delta p}\right)^{1/p}\left(\frac{1}{|Q|}\int_Q v^{\delta p'}\right)^{1/p'}\\
	&\leq [u^{\delta p}]_{\mathrm{RH}_{1/\delta}^{\rho,\theta_1}}^{1/p}[v^{\delta p'}]_{\mathrm{RH}_{1/\delta}^{\rho,\theta_2}}^{1/p'}\left(1+\frac{r}{\rho(x)}\right)^{\theta_1+\theta_2+\theta_3}\left(\frac{1}{|Q|}\int_Q (uv)^{\delta}\right)\\
	&\leq [u^{\delta p}]_{\mathrm{RH}_{1/\delta}^{\rho,\theta_1}}^{1/p}[v^{\delta p'}]_{\mathrm{RH}_{1/\delta}^{\rho,\theta_2}}^{1/p'}\left(1+\frac{r}{\rho(x)}\right)^{\theta_1+\theta_2+\theta_3}\left(\frac{1}{|Q|}\int_Q uv\right)^{\delta}.\\
	\end{align*}	
	By raising both sides to the power $1/\delta$ we can conclude the thesis.
\end{proof}

As a corollary, we have that the product of two weights in $\mathrm{RH}_\infty^\rho$ is a weight in the same class.

\begin{cor}\label{coro: producto de pesos en RHinf,rho esta en RHinf,rho}
	Let $u$ and $v$ be two weights belonging to the $\mathrm{RH}_\infty^\rho$ class. Then  $uv\in \mathrm{RH}_\infty^\rho$.
\end{cor}

\begin{proof}
	The hypothesis implies that both $u$ and $v$ belong to the $\mathrm{RH}_2^\rho$ class. By applying Lemma~\ref{lema: Holder al reves con producto de pesos} with $p=p'=2$ we get
	\begin{align*}
	u(x)v(x)&\leq [u]_{\mathrm{RH}_\infty^{\rho,\theta_1}}[v]_{\mathrm{RH}_\infty^{\rho,\theta_2}}\left(1+\frac{r}{\rho(x)}\right)^{\theta_1+\theta_2}\left(\frac{1}{|Q|}\int_Q u\right)\left(\frac{1}{|Q|}\int_Q v\right)\\
	&\leq [u]_{\mathrm{RH}_\infty^{\rho,\theta_1}}[v]_{\mathrm{RH}_\infty^{\rho,\theta_2}}\left(1+\frac{r}{\rho(x)}\right)^{\theta_1+\theta_2}\left(\frac{1}{|Q|}\int_Q u^2\right)^{1/2}\left(\frac{1}{|Q|}\int_Q v^2\right)^{1/2}\\
	&\leq C[u]_{\mathrm{RH}_\infty^{\rho,\theta_1}}[v]_{\mathrm{RH}_\infty^{\rho,\theta_2}}\left(1+\frac{r}{\rho(x)}\right)^{\theta_1+\theta_2+\theta_3}\left(\frac{1}{|Q|}\int_Q uv\right).\qedhere
	\end{align*}
\end{proof}

\medskip

The following lemma establishes an important relation between the classes $A_p^\rho(u)$ and $A_p^\rho$. The corresponding classical version is given in \cite{CruzUribe-Martell-Perez}.

\begin{lem}\label{lema: u en A1,rho y v en Ap,rho implican uv en Ap,rho}
	Let $1\leq p\leq \infty$. If $u\in A_1^\rho$ and $v\in A_p^\rho(u)$ then $uv\in A_p^{\rho}$.
\end{lem}

\begin{proof}
	We shall first consider the case $p=1$. Since $v\in A_1^\rho(u)$, we can find $\theta_1\geq 0$ for which the inequality
	\[\frac{1}{u(Q)}\int_Q vu\leq [v]_{A_1^{\rho,\theta_1}}\left(1+\frac{r}{\rho(x)}\right)^{\theta_1}\inf_Q v\]
	holds for every cube $Q$. Similarly, $u\in A_1^\rho$ implies that there exists $\theta_2\geq 0$ such that
	\[\frac{1}{|Q|}\int_Q u\leq [u]_{A_1^{\rho,\theta_2}}\left(1+\frac{r}{\rho(x)}\right)^{\theta_2}\inf_Q u,\]
	for every cube $Q$.
	Therefore,
	\begin{align*}
	\frac{1}{|Q|}\int_Q uv &=\frac{u(Q)}{|Q|}\frac{1}{u(Q)}\int_Q uv\\
	&\leq [u]_{A_1^{\rho,\theta_2}}\left(1+\frac{r}{\rho(x)}\right)^{\theta_2}\left(\inf_Q u\right)[v]_{A_1^{\rho,\theta_1}}\left(1+\frac{r}{\rho(x)}\right)^{\theta_1}\inf_Q v\\
	&\leq  [u]_{A_1^{\rho,\theta_2}}[v]_{A_1^{\rho,\theta_1}}\left(1+\frac{r}{\rho(x)}\right)^{\theta_1+\theta_2}\inf_Q uv,
	\end{align*}
	which means that $uv\in A_1^\rho$.
	
	We turn now to the case $1<p<\infty$. If $v\in A_p^{\rho,\theta_1}$, then
	\begin{align*}
	\left(\frac{1}{|Q|}\int_Q uv\right)^{1/p}\left(\frac{1}{|Q|}\int_Q (uv)^{1-p'}\right)^{1/p'}&=\frac{u(Q)}{|Q|}\left(\frac{1}{u(Q)}\int_Q uv\right)^{1/p}\left(\frac{1}{u(Q)}\int_Q v^{1-p'}uu^{-p'}\right)^{1/p'}\\
	&\leq [u]_{A_1^{\rho,\theta_2}}\left(1+\frac{r}{\rho(x)}\right)^{\theta_2}\left(\frac{1}{u(Q)}\int_Q uv\right)^{1/p}\\
	&\qquad\times\left(\frac{1}{u(Q)}\int_Q v^{1-p'}u\right)^{1/p'}\\
	&\leq [u]_{A_1^{\rho,\theta_2}}[v]_{A_p^{\rho,\theta_1}(u)}\left(1+\frac{r}{\rho(x)}\right)^{\theta_1+\theta_2},
	\end{align*}
	which gives us the desired estimate.
	
	Finally, if $p=\infty$ there exists $1<q<\infty$ such that $v\in A_q^\rho(u)$. By applying the previous case we get $uv\in A_q^\rho$ and so $uv\in A_\infty^\rho.$
\end{proof}

We are now in a position to state and prove a result that will play a fundamental role in our estimates.

\begin{thm}\label{teo: lema fundamental}
	Let $u\in A_1^\rho$ and $v$ a weight such that $uv\in A_\infty^\rho$. Then $u\in A_1^\rho(v)$, that is, there exist $C>0$ and $\theta\geq 0$ such that 
	\[\frac{uv(Q)}{v(Q)}\leq C\left(1+\frac{r}{\rho(x)}\right)^{\theta}\inf_Q u,\]
	for every cube $Q=Q(x,r)$.
\end{thm}

\begin{proof}
	Since $uv\in A_\infty^\rho$, there exist $1<p,s<\infty$ verifying $uv\in A_p^\rho\cap\mathrm{RH}_s^\rho$. By Proposition~\ref{propo: pesos de Ainf,rho se factorizan}, $uv=w_1w_2$, where $w_1\in A_1^\rho\cap \mathrm{RH}_s$ and $w_2\in A_p^\rho\cap \mathrm{RH}_\infty^\rho$. On the other hand, $u\in A_1^\rho$ implies that $u^{-1}\in \mathrm{RH}_\infty^\rho$, by virtue of Lemma~\ref{lema: w en A1,rho implica w^{1-p} en RHinfty,rho} with $p=2$. Corollary~\ref{coro: producto de pesos en RHinf,rho esta en RHinf,rho} allows us to conclude that $w_2u^{-1}\in \mathrm{RH}_\infty^\rho$. Given a cube $Q=Q(x,r)$ we have that
	\begin{align*}
	\frac{uv(Q)}{v(Q)}&=\frac{w_1w_2(Q)}{w_1w_2u^{-1}(Q)}\\
	&\leq \frac{1}{\inf_Q w_1}\frac{|Q|}{w_2u^{-1}(Q)}\left(\frac{1}{|Q|}\int_Q w_1^s\right)^{1/s}\left(\frac{1}{|Q|}\int_Q w_2^{s'}\right)^{1/s'}\\
	&\leq \frac{1}{\inf_Q w_1} \frac{[w_2u^{-1}]_{\mathrm{RH}_\infty^{\rho,\theta_1}}}{\sup_Q (w_2u^{-1})}\left(1+\frac{r}{\rho(x)}\right)^{\theta_1+\theta_2+\theta_3}[w_1]_{\mathrm{RH}_s^{\rho,\theta_2}}\frac{w_1(Q)}{|Q|}[w_2]_{\mathrm{RH}_{s'}^{\rho,\theta_3}}\frac{w_2(Q)}{|Q|}\\
	&\leq [w_1]_{A_1^{\rho,\theta_4}} [w_2u^{-1}]_{\mathrm{RH}_\infty^{\rho,\theta_1}}\left(1+\frac{r}{\rho(x)}\right)^{\theta_1+\theta_2+\theta_3+\theta_4}[w_1]_{\mathrm{RH}_s^{\rho,\theta_2}}[w_2]_{\mathrm{RH}_{s'}^{\rho,\theta_3}}\left(\frac{1}{|Q|}\int_Q \frac{w_2u^{-1}u}{\sup_Q (w_2u^{-1})}\right)\\
	&\leq [u]_{A_1^{\rho,\theta_5}} [w_1]_{A_1^{\rho,\theta_4}} [w_2u^{-1}]_{\mathrm{RH}_\infty^{\rho,\theta_1}}\left(1+\frac{r}{\rho(x)}\right)^{\theta_1+\theta_2+\theta_3+\theta_4+\theta_5}[w_1]_{\mathrm{RH}_s^{\rho,\theta_2}}[w_2]_{\mathrm{RH}_{s'}^{\rho,\theta_3}}\inf_Q u,
	\end{align*}
	which leads us to the desired estimate.
\end{proof}

An immediate consequence of this result is the following.

\begin{cor}\label{cor: lema fundamental}
	Let $u\in A_1^\rho$ and $v\in A_\infty^{\rho}(u)$. Then, there exist $C>0$ and $\theta\geq 0$ such that for every cube $Q$ the inequality
	\[\frac{uv(Q)}{v(Q)}\leq C\left(1+\frac{r}{\rho(x)}\right)^{\theta}\inf_Q u\]
	holds.
\end{cor}

\begin{proof}
	It follows straightforwardly by combining Lemma~\ref{lema: u en A1,rho y v en Ap,rho implican uv en Ap,rho} with Theorem~\ref{teo: lema fundamental}.
\end{proof}

\begin{prop}\label{propo: v en A1(u) implica v en A1(u^alpha)}
	Let $0<\alpha<1$, $1<p\leq \infty$, $u\in A_1^\rho$ and $v\in A_p^\rho(u)$. Then $v\in A_p^{\rho}(u^{\alpha})$.
\end{prop}

\begin{proof}
	Let us first assume that $p<\infty$. By Corollary~\ref{coro: potencias negativas de A1,rho estan en RHinfty,rho} we have that $u^{\alpha-1}\in \mathrm{RH}_\infty^\rho$, which implies that there exists $\theta_1\geq 0$ such that $u^{\alpha-1}\in \mathrm{RH}_\infty^{\rho,\theta_1}$. On the other hand, there exist nonnegative constants $\theta_2$ and $\theta_3$ verifying $u\in A_1^{\rho,\theta_2}$ and $v\in A_p^{\rho,\theta_3}(u)$. Fixed a cube $Q$, we have
	\begin{align*}
	\frac{1}{u^{\alpha}(Q)}\int_Q vu^{\alpha}&
	\leq \left(\sup_Q u^{\alpha-1}\right)\frac{u(Q)}{|Q|}\frac{|Q|}{u^{\alpha}(Q)}\frac{1}{u(Q)}\int_Q vu\\
	&\leq \left[u^{\alpha-1}\right]_{\mathrm{RH}_{\infty}^{\rho,\theta_1}}[u]_{A_1^{\rho,\theta_2}}\frac{1}{|Q|}\int_Q u^{\alpha-1}\left(\inf_Q u\right)\frac{|Q|}{u^{\alpha}(Q)}\left(\frac{1}{u(Q)}\int_Q vu\right)\left(1+\frac{r}{\rho(x)}\right)^{\theta_1+\theta_2}\\
	&=\left[u^{\alpha-1}\right]_{\mathrm{RH}_{\infty}^{\rho,\theta_1}}[u]_{A_1^{\rho,\theta_2}}\left(\frac{1}{u(Q)}\int_Q vu\right)\left(1+\frac{r}{\rho(x)}\right)^{\theta_1+\theta_2}.
	\end{align*}
	Therefore,
	\[\left(\frac{1}{u^{\alpha}(Q)}\int_Q vu^{\alpha}\right)^{1/p}\leq \left[u^{\alpha-1}\right]_{\mathrm{RH}_{\infty}^{\rho,\theta_1}}^{1/p}[u]_{A_1^{\rho,\theta_2}}^{1/p}\left(\frac{1}{u(Q)}\int_Q vu\right)^{1/p}\left(1+\frac{r}{\rho(x)}\right)^{\frac{\theta_1+\theta_2}{p}}.\]
	By following the same argument we can obtain
	\begin{align*}
	\frac{1}{u^\alpha(Q)}\int_Qv^{1-p'}u^\alpha&\leq \left(\sup_Q u^{\alpha-1}\right)\frac{u(Q)}{|Q|}\frac{|Q|}{u^{\alpha}(Q)}\frac{1}{u(Q)}\int_Q v^{1-p'}u\\
	&\leq \left[u^{\alpha-1}\right]_{\mathrm{RH}_{\infty}^{\rho,\theta_1}}[u]_{A_1^{\rho,\theta_2}}\left(\frac{1}{u(Q)}\int_Q v^{1-p'}u\right)\left(1+\frac{r}{\rho(x)}\right)^{\theta_1+\theta_2},
	\end{align*}
	which yields
	\[\left(\frac{1}{u^\alpha(Q)}\int_Qv^{1-p'}u^\alpha\right)^{1/p'}\leq \left[u^{\alpha-1}\right]_{\mathrm{RH}_{\infty}^{\rho,\theta_1}}^{1/p'}[u]_{A_1^{\rho,\theta_2}}^{1/p'}\left(\frac{1}{u(Q)}\int_Q v^{1-p'}u\right)^{1/p'}\left(1+\frac{r}{\rho(x)}\right)^{\frac{\theta_1+\theta_2}{p'}}.\]
	By combining these two estimates we get
	\[\left(\frac{1}{u^{\alpha}(Q)}\int_Q vu^{\alpha}\right)^{1/p}\left(\frac{1}{u^\alpha(Q)}\int_Qv^{1-p'}u^\alpha\right)^{1/p'}\leq \left[u^{\alpha-1}\right]_{\mathrm{RH}_{\infty}^{\rho,\theta_1}}[u]_{A_1^{\rho,\theta_2}}[v]_{A_p^{\rho,\theta_3}(u)}\left(1+\frac{r}{\rho(x)}\right)^{\theta_1+\theta_2+\theta_3},\]
	which implies that $v\in A_p^{\rho}(u^\alpha)$. 
	
	If $p=\infty$, there exists $q>1$ that verifies $v\in A_q^\rho(u)$. By the previous case, $v\in A_q^\rho(u^\alpha)\subseteq A_\infty^\rho(u^{\alpha})$. \qedhere
\end{proof}

The next result will be useful in the proof of Theorem~\ref{thm: mixta para T SCZO (s,delta)}.

\begin{prop}\label{prop: peso para el tipo fuerte de T (s,delta)}
	Let $1<s<\infty$ be fixed, $u^{s'}\in A_1^\rho$ and $v$ a weight that verifies $v\in A_\infty^\rho(u^{\beta})$, for some $\beta>s'$. Then there exists a number $q$, $1<q<s$, such that $u^{1-q'}v\in A_{q'/s'}^\rho$.
\end{prop}

\begin{proof}
	By virtue of Remark~\ref{obs: existen potencias mayores que 1 de A_1,rho en A1,rho} we can find $\varepsilon_0$ such that $u^{s'(1+\varepsilon)}\in A_1^\rho$, for every $0<\varepsilon<\varepsilon_0$. Since $v\in A_\infty^\rho(u^{\beta})$, there exists $r>1$ such that $v\in A_r^{\rho}(u^{\beta})$. 
	Let us fix $\varepsilon$ such that 
	\[0<\varepsilon<\min\left\{\frac{\beta}{s'}-1,\frac{1}{s(r-1)},\varepsilon_0\right\},\]
	and define $q$ such that 
	\[q'=\frac{s'(1+\varepsilon)-1}{\varepsilon}.\]
	Therefore we have that $q'>s'>1$, or equivalently, $1<q<s$. Let $\alpha=s'(q'-1)/(q'-s')$. Then, by our choice of $\varepsilon$ we can conclude that $u^\alpha\in A_1^{\rho,\theta_1}$, for some $\theta_1\geq 0$. Indeed, by the definition of $q'$ we have
	\[q'-s'=\frac{s'-1}{\varepsilon}.\]
	So we can obtain 
	\begin{align*}
	\frac{s'(q'-1)}{q'-s'}&=\frac{s'\varepsilon}{s'-1}\left(\frac{s'(1+\varepsilon)-1}{\varepsilon}-1\right)\\
	&=\frac{s'}{s'-1}\left(s'-1+\varepsilon(s'-1)\right)\\
	&=\frac{s'(s'-1)(1+\varepsilon)}{s'-1}\\
	&=s'(1+\varepsilon).
	\end{align*}
	Let us now check that $u^{1-q'}v\in A_{q'/s'}^\rho$. Notice first that
	\[\frac{q'}{s'}=\frac{1+\varepsilon}{\varepsilon}-\frac{1}{s'\varepsilon}=1+\frac{1}{s\varepsilon},\]
	and again $r<q'/s'$ by the choice of $\varepsilon$. Combining this fact with Proposition~\ref{propo: v en A1(u) implica v en A1(u^alpha)} applied with $\alpha/\beta<1$ we conclude that $v\in A_{q'/s'}^{\rho,\theta_2}(u^{\alpha})$, for some $\theta_2\geq 0$. Given a cube $Q$ of $\mathbb{R}^d$, since $(1-q'-\alpha)s'/q'=-\alpha$ we obtain
	\begin{align*}
	\left(\frac{1}{|Q|}\int_Q u^{1-q'}v\right)^{s'/q'}\left(\frac{1}{|Q|}\int_Q \left(u^{1-q'}v\right)^{1-(q'/s')'}\right)^{1-s'/q'}&=\left(\frac{u^{\alpha}(Q)}{|Q|}\right)\left(\frac{1}{u^{\alpha}(Q)}\int_Q u^{1-q'-\alpha}vu^{\alpha}\right)^{s'/q'}\\
	&\qquad\times\left(\frac{1}{u^{\alpha}(Q)}\int_Q v^{1-(q'/s')'}u^{\alpha}\right)^{1-s'/q'}\\
	&\leq [v]_{A_{q'/s'}^{\rho,\theta_2}(u^{\alpha})}[u^\alpha]_{A_1^{\rho,\theta_1}}\left(\sup_Q u^{-\alpha}\right)\left(\inf_Q{u^\alpha}\right)\\
	&\qquad \times \left(1+\frac{r}{\rho(x)}\right)^{\theta_1+\theta_2}\\
	&\leq [v]_{A_{q'/s'}^{\rho,\theta_2}(u^{\alpha})}[u^\alpha]_{A_1^{\rho,\theta_1}}\left(1+\frac{r}{\rho(x)}\right)^{\theta_1+\theta_2},
	\end{align*}
	 and thus $u^{1-q'}v\in A_{q'/s'}^{\rho,\theta_1+\theta_2}\subseteq A_{q'/s'}^{\rho}$. \qedhere
\end{proof}

%

\section{Proof of the main results}\label{seccion: prueba de resultados principales}
We devote this section to prove our main theorems established in \S~\ref{seccion: introduccion}. Before proving them we give some auxiliary definitions and results which will be needed in the sequel.

We shall very often refer to \emph{critical cubes}, meaning cubes of the type $Q(x_0,\rho(x_0))$, and we call \emph{subcritical cubes} to those $Q(x_0,r)$ with $r\leq \rho(x_0)$. The family of all subcritical cubes will be denoted by $\mathcal{Q}_\rho$.  Observe that from \eqref{eq-constantesRho}, $\rho(y)\simeq\rho(x_0)$ whenever $y\in Q(x_0,\rho(x_0))$. The following result is a useful consequence of \eqref{eq-constantesRho} and can be found in~\cite{DZ-99}.

\begin{prop} \label{prop-cubrimientocritico}
	There exists a sequence of points $\{x_j\}_{j\in\mathbb{N}}$ such that the family of critical cubes given by $Q_j=~Q(x_j,\rho(x_j))$  satisfies
	\begin{enumerate}[\rm(a)]
		\item \label{item: prop-cubrimientocritico - item a}$\displaystyle \bigcup_{j\in\mathbb{N}} Q_j= \mathbb{R}^d$.
		\item \label{item: prop-cubrimientocritico - item b}There exist positive constants $C$ and $N_1$ such that for any $\sigma\geq1$,
		$\displaystyle \sum_{j\in\mathbb{N}}\mathcal{X}_{\sigma Q_j}\leq C \sigma^{N_1}$.
	\end{enumerate}
\end{prop}

We shall require the following Calderón-Zygmund decomposition on a fixed cube (see, for example, \cite{Aimar85}).

\begin{lem}\label{lema: DCZ en un cubo fijo}
	Let $R$ be a cube in $\mathbb{R}^d$, $v$ a doubling weight and a function $f\in L^1(R,v\,dx)$. Then for $t>\frac{1}{v(R)}\int_R f(x)v(x)\,dx$ there exist functions $g$ and $h$ and a collection of dyadic sub-cubes $\{P_i\}_{i\in\mathbb{N}}$ of $R$ such that
	\begin{enumerate}[\rm(a)]
		\item \label{item: lema: DCZ en un cubo fijo - item a}$f=g+h$;
		\item \label{item: lema: DCZ en un cubo fijo - item b}$|g(x)|\leq C t$ for almost every $x$;
		\item \label{item: lema: DCZ en un cubo fijo - item c}$h=\sum_{i\in\mathbb{N}} h_i$ where each $h_i$ is supported on a dyadic cube $P_i$. Furthermore, these cubes are pairwise disjoint and
		\[\int_{P_i}h_i(x)v(x)\,dx=0.\]
	\end{enumerate} 
\end{lem}

We shall be dealing with local versions of the classical Hardy-Littlewood maximal operator. For a fixed cube $R\subseteq \mathbb{R}^d$ we define
\begin{equation}\label{eq: maximal clasica localizada}
M_Rf(x)=\sup_{Q\ni x, Q\subseteq R} \frac{1}{|Q|}\int_Q |f(y)|\,dy.
\end{equation}
We shall also consider a dyadic version of the operator above. In order to define it we shall introduce the following concept.

A \emph{dyadic grid} $\mathcal{D}$ will be understood as a collection of cubes in $\mathbb{R}^d$ with the following properties:
\begin{enumerate}
	\item every cube  $Q$ in $\mathcal{D}$ verifies $\ell(Q)=2^k$, for some $k\in\mathbb{Z}$;
	\item if $P$ and $Q$ are in $\mathcal{D}$ and $P\cap Q\neq\emptyset$, then either $P\subseteq Q$ or $Q\subseteq P$;
	\item $\mathcal{D}_k=\{Q\in \mathcal{D}: \ell(Q)=2^k\}$ is a partition of $\mathbb{R}^d$, for every $k\in \mathbb{Z}$.
\end{enumerate}

Given a dyadic grid $\mathcal{D}$ and a cube $R$, by $\mathcal{D}_R$ we shall understand the collection of cubes in the grid that are also contained in $R$, that is
\[\mathcal{D}_R=\{Q\in\mathcal{D}: Q\subseteq R\}.\]
We also denote by $\mathscr{D}(R)$ the collection of dyadic subcubes of $R$. Notice that if $R$ is a dyadic cube in $\mathcal{D}$, then $\mathcal{D}_R=\mathscr{D}(R)$. 
The  weighted dyadic version of \eqref{eq: maximal clasica localizada} is given by
\begin{equation}\label{eq: maximal diadica localizada}
M_{R,w}^{\mathscr{D}}f(x)=\sup_{Q\ni x, Q\in \mathscr{D}(R)} \frac{1}{w(Q)}\int_Q |f(y)|w(y)\,dy.
\end{equation}
When $w=1$ we simply write $M_{R}^{\mathscr{D}}$.


The following lemma establishes an important geometric relation between cubes in $\mathbb{R}^d$ and dyadic grids (see \cite{Okikiolu}).

\begin{lem}\label{lema: cubrimiento por grillas diadicas}
	There exists dyadic grids $\mathcal{D}^{(i)}$, $1\leq i\leq 3^d$, such that for every cube $Q$ in $\mathbb{R}^d$ there exists an index $1\leq i_0\leq 3^d$ and a dyadic cube $Q_0\in \mathcal{D}^{(i_0)}$ such that $Q\subseteq Q_0$ and $\ell(Q_0)\leq 3\ell(Q)$. 
\end{lem} 

The result above allows us to prove a useful relation between the localized maximal functions given above.

\begin{lem}\label{lem-Mloc a Mdiad}
	There exists dyadic grids $\mathcal{D}^{(i)}$, $1\leq i\leq 3^d$ with the following property:
	For every cube $Q$ in $\mathbb{R}^d$ there exists $3^d$ dyadic cubes $Q_i\in \mathcal{D}^{(i)}$, $1\leq i\leq 3^d$ such that
	\[M_Q f(x)\leq 3^d \sum_{i=1}^{3^d}M_{Q_i}^{\mathscr{D}}(f\mathcal{X}_{Q})(x),\]
	for every $x\in Q$. Furthermore, each $Q_i$ verifies $Q_i\subseteq \lambda Q$, where $\lambda$ depends only on $d$.
\end{lem}

\begin{proof}
	Before proceeding with the proof, given a cube $Q_0$ and a dyadic grid $\mathcal{D}$ we shall denote by $M_{Q_0,\mathcal{D}}f$ the version of $M_{Q_0}$ where the supremum is taken only for cubes in $\mathcal{D}$ contained in $Q_0$, that is
	\[M_{Q_0,\mathcal{D}}f(x)=\sup_{Q\in \mathcal{D}_{Q_0}}\frac{1}{|Q|}\int_Q |f|.\]
	
	Fix $x\in Q$ and let $P\subseteq Q$ be a subcube containing $x$. By Lemma~\ref{lema: cubrimiento por grillas diadicas} there exists a dyadic grid $\mathcal{D}^{(i)}$ and $P_i\in\mathcal{D}^{(i)}$ such that $P\subseteq P_i$ and $\ell(P_i)\leq 3\ell(P)$. We claim that $P_i\subseteq 8\sqrt{d}Q$. Indeed, if $x_Q$ denotes the center of $Q$, for $y\in P_i$ we have that
	\[|y-x_Q|\leq |y-x|+|x-x_Q|\leq \sqrt{d}\ell(P_i)+\frac{\sqrt{d}}{2}\ell(Q)\leq 4\sqrt{d}\ell(Q),\]
	so $P_i\subseteq B(x_Q, 4\sqrt{d}\ell(Q))\subseteq 8\sqrt{d}Q$.
	Therefore we have that
	\begin{align*}
	\frac{1}{|P|}\int_P |f|&=\frac{1}{|P|}\int_P |f|\mathcal{X}_Q\leq \frac{3^d}{|P_i|}\int_{P_i}|f|\mathcal{X}_Q	\\
	&\leq 3^d M_{ 8\sqrt{d}Q,\mathcal{D}^{(i)}}(f\mathcal{X}_Q)(x)\\
	&\leq 3^d\sum_{i=1}^{3^d}M_{ 8\sqrt{d}Q,\mathcal{D}^{(i)}}(f\mathcal{X}_Q)(x). 
	\end{align*}
By taking supremum over the cubes $P\subseteq Q$ we arrive to
\[M_{Q}f(x)\leq 3^d\sum_{i=1}^{3^d}M_{ 8\sqrt{d}Q,\mathcal{D}^{(i)}}(f\mathcal{X}_Q)(x),\]
for every $x\in Q$.

    Fix $1\leq i\leq 3^d$. There exists a unique $k\in\mathbb{Z}$ such that
    \begin{equation}\label{eq: lem-Mloc a Mdiad - eq1}
     2^k<8\sqrt{d}\ell(Q)\leq 2^{k+1}.   
    \end{equation}
    There also exist at most $2^d$ cubes in $\mathcal{D}^{(i)}$ with side length $2^k$ and that intersect $8\sqrt{d}Q$. Let $Q_i$ be the cube in $\mathcal{D}^{(i)}_{k+1}$ that contains these cubes. We claim that
    \[8\sqrt{d}Q\subseteq Q_i\subseteq 48d \,Q.\]
    Indeed, the first inclusion is immediate. For the latter, if $y\in Q_i$ and $x\in 8\sqrt{d}Q\cap Q_i$ by \eqref{eq: lem-Mloc a Mdiad - eq1} we have that
    \begin{align*}
    |x_Q-y|&\leq |x_Q-x|+|x-y|\\
    &\leq 8d\,\ell(Q)+\sqrt{d}\ell(Q_i)\\
    &\leq 8d\,\ell(Q)+16d\,\ell(Q)\\
    &= 24d\,\ell(Q),
    \end{align*}
    so $Q_i\subseteq B(x_Q, 24 d\ell(Q))\subseteq 48d\,Q.$
    
    By our choice of $Q_i$, $1\leq i\leq 3^d$, we must have
	\[M_{ 8\sqrt{d}Q,\mathcal{D}^{(i)}}(f\mathcal{X}_Q)\leq M_{ Q_i,\mathcal{D}^{(i)}}(f\mathcal{X}_Q)=M_{Q_i}^{\mathscr{D}}(f\mathcal{X}_Q).\]
	Finally,
	\[M_Qf(x)\leq 3^d\sum_{i=1}^{3^d}M_{Q_i}^{\mathscr{D}}(f\mathcal{X}_Q)(x),\]
	as desired.
\end{proof}
In the sequel we shall consider local classes of weights associated to a critical radius function as it was done in~\cite{BHS-classesofweights}. Given a weight $u$ and a critical radius function $\rho$, we say that a weight $w\in A^{\rho,\textup{loc}}_p(u)$ if it satisfies ~\eqref{eq: clase Ap,rho,theta(u)} for each cube $Q\in \mathcal{Q}_\rho$, and $w\in A^{\rho,\textup{loc}}_1(u)$ if it satisfies~\eqref{eq: clase A1,rho,theta(u)} for each cube $Q\in \mathcal{Q}_\rho$.
We also  consider classes of weights associated to a cube $R$. More precisely we  say that a weight $w\in A_p(R,u)$ if it satisfies~\eqref{eq: clase Ap,rho,theta(u)} for each cube $Q\subseteq R$, and $w\in A_1(R,u)$ if it verifies~\eqref{eq: clase A1,rho,theta(u)} for each cube $Q\subseteq R$. In all cases we will drop $u$ from the notation when $u=1$. It is easy to check that properties stated in Lemma~\ref{lema: propiedades clase Ap,rho} hold for these classes of weights.

It is well known that the classes of weights mentioned above are related with the continuity properties of the localized Hardy-Littlewood maximal operators $M_R$ (see~\cite{MS81} and~\cite{Cald76}). 

The following proposition, that  can be viewed as a localized version of Theorem~1.4 in~\cite{CruzUribe-Martell-Perez}, will play an important role in proving Theorem~\ref{thm: mixta para M}. 

\begin{prop}\label{Prop-maxRdiad}
	Let $R$ be a cube, $u\in A_1(R)$ and $v\in A_\infty(R,u)$. There exists a positive constant $C$ depending only on the characteristic constant of the weights and the dimension,  such that for all $t>0$
	\begin{equation*}
		uv\left(\left\{ x\in R: \frac{M_R^\mathscr{D}(fv)(x)}{v(x)}>t
		\right\}\right)
		\leq \frac{C}{t}\int_R |f(x)|u(x)v(x)dx.
	\end{equation*}
\end{prop}

\begin{proof}
    It is enough to consider a nonnegative  function $f$ and $t> f^{uv}_R = \frac{1}{uv(R)}\int_R f\, u\, v$, since in other case the estimate is trivial.
	We first apply Lemma~\ref{lema: DCZ en un cubo fijo} to decompose $R$ with respect to the measure $d\mu(x) = u(x)v(x)dx$  at height $t$, obtaining a sequence of pairwise disjoint dyadic cubes $\{P_i\}_{i\in\mathbb{N}}$ such that 
	\begin{equation*}
		t \leq f^{uv}_{P_i}\leq \gamma t,
	\end{equation*}
	and $f(x)\leq t$ for $x\in R\setminus \Omega$, being $\Omega = \bigcup_{i\in\mathbb{N}} P_i$, where $\gamma = \gamma(d)>1$.
	
	Therefore, we can write $f = g + h$, where 
	\begin{equation*}
	g(x) = 
	\begin{cases} 
		f^{uv}_{P_i}          & \mbox{if } x\in P_i   \\
		f(x)  & \mbox{if } x\in R\setminus\Omega,
	\end{cases}
\end{equation*}
and $h(x) = \sum h_i(x)$ where
\begin{equation*}
	h_i(x) = (f - f^{uv}_{P_i}) \mathcal{X}_{P_i}.
\end{equation*}

Let $Q\in \mathscr{D}(R)$, then for all $x\in Q$,
\begin{equation*}
	\begin{split}
	\frac{1}{|Q|} \int_Q fv  \leq 
	\frac{[u]_{A_1(R)}}{u(Q)} \int_Q fuv  \leq [u]_{A_1(R)}(M^\mathscr{D}_{R,u} (gv)(x) + \tilde{M}_u(hv)(x)),
\end{split}
\end{equation*}
since $u\in A_1(R)$, where
\begin{equation*}
		\tilde{M}_u \phi(x) =  \sup_{\substack{Q\ni x\\Q\in \mathscr{D}(R)}} \left| \frac{1}{u(Q)} \int_Q \phi u \right|.
\end{equation*}
Therefore, 
\begin{equation*}
	M_R^\mathscr{D} (fv)(x) \leq M^\mathscr{D}_{R,u}(gv)(x) + \tilde{M}_u(hv)(x).
\end{equation*}
Now,
\begin{equation*}
	\begin{split}
		uv\left( \left\{ x\in R: \frac{M^\mathscr{D}_R(fv)(x)}{v(x)} >t \right\} \right)
		 & \leq 
		 uv\left( \left\{ x\in R: \frac{M^\mathscr{D}_{R,u}(gv)(x)}{v(x)} > \frac{t}{2} \right\} \right)
		 \\ & \hspace{5pt} + 
		 uv\left( \left\{ x\in \Omega: \frac{\tilde{M}_u(hv)(x)}{v(x)} > \frac{t}{2} \right\} \right)
		 \\ & \hspace{5pt} + 
		 uv\left( \left\{ x\in R\setminus \Omega: \frac{\tilde{M}_u(hv)(x)}{v(x)} > \frac{t}{2} \right\} \right)
		 \\ & = I_1 + I_2 + I_3.
	\end{split}
\end{equation*}

Let us take care of $I_1$. By Chebyshev inequality, the continuity properties of the localized Hardy-Littlewood maximal operator with weights and the properties of $g$ we have  

\begin{equation*}
	\begin{split}
		I_1  & \leq 
		\frac{C}{t^p} \int_R M^{\mathscr{D}}_{R,u}(gv)(x)^p u(x) v(x)^{1-p}dx 
		\\ & \leq 
		\frac{C}{t^p} \int_R g(x)^p u(x) v(x)dx 
		\\ & \leq \frac{C \gamma^{p-1}}{t} \int_R  g(x) u(x) v(x)dx  
		\\ & =  \frac{C }{t} \int_{R\setminus \Omega}  f(x)u(x) v(x)dx + \frac{C}{t} \sum_{i\in\mathbb{N}}\int_{P_i} f(x)u(x)v(x) dx 
		\\ & = \frac{C}{t} \int_R f(x)u(x)v(x)dx.
	\end{split}
\end{equation*}

To deal with $I_2$,
\[I_2  \leq uv(\Omega) = \sum_{i\in\mathbb{N}} uv(P_i) \leq \frac{1}{t}\sum_{i\in\mathbb{N}}
		\int_{P_i} f(x)u(x)v(x)\,dx
		 \leq \frac{C}{t} \int_R f(x)u(x)v(x)\,dx.\]
		

Finally, by following a similar argument as in Theorem 1.4 in~\cite{CruzUribe-Martell-Perez} we can show that $I_3 = 0$. 
\end{proof}

We are now in a position to prove Theorem~\ref{thm: mixta para M}. 

\begin{proof}[Proof of Theorem~\ref{thm: mixta para M}]

	Let us begin by observing that, for $\varphi\in L^1_{\textup{loc}}(\RR^d)$ and $\sigma>0$, 
	\begin{equation*}
		\begin{split}
			M^{\rho,\sigma} \varphi (x)  & \leq 
			\sup_{\substack{Q\ni x \\ Q \in \mathcal{Q}_\rho}} \frac{1}{|Q|} \int_Q |\varphi| + 
			\sup_{\substack{Q\ni x \\ Q \notin \mathcal{Q}_\rho}} \left(\frac{\rho(x_Q)}{r_Q}\right)^{\sigma} \frac{1}{|Q|} \int_Q |\varphi|
		\\	& = M^\rho_{\textup{loc}}  \varphi (x) +  M^{\rho,\sigma}_{\textup{glob}}  \varphi (x).
			\end{split}
	\end{equation*}
	We can assume without losing generality that $f$ is nonnegative. Let $\{Q_j\}_{j\in\mathbb{N}}$ the family of critical cubes given in Proposition~\ref{prop-cubrimientocritico} and satisfying \ref{item: prop-cubrimientocritico - item a}. For $t>0$ we write
	\begin{equation*}
	\begin{split}
	uv\left(\left\{x\in\mathbb{R}^d:\frac{M^{\rho,\sigma}(fv)(x)}{v(x)}>t\right\}\right)
	& \leq \sum_{j\in\mathbb{N}} uv\left(\left\{x\in Q_j:\frac{M^{\rho,\sigma}(fv)(x)}{v(x)}>t\right\}\right)
	\\ & \leq  \sum_{j\in\mathbb{N}} uv\left(\left\{x\in Q_j:\frac{M^{\rho,\sigma}_{\textup{glob}}(f v)(x)}{v(x)}>\frac{t}{2}\right\}\right)
	\\ & +  \sum_{j\in\mathbb{N}} uv\left(\left\{x\in Q_j:\frac{M^\rho_{\textup{loc}}(fv)(x)}{v(x)}>\frac{t}{2}\right\}\right).
	\end{split}
	\end{equation*}	
	
	We will consider first the term corresponding to $M^{\rho,\sigma}_{\textup{glob}}$. Observe that for $x\in Q_j$, there exist a constant $c>0$ such that
	\begin{equation*}
	\begin{split}
	M^{\rho,\sigma}_{\textup{glob}}(fv)(x) & \leq C \sup_{k\geq 1} \frac{2^{-c k\sigma}}{|2^k Q_j|}\int_{2^kQ_j}  f v
	\\ & \leq C \sum_{k\geq 1} \frac{2^{-k(c\sigma-\theta)}}{u(2^kQ_j)} \int_{2^kQ_j} f v u 
	\\ & \leq \frac{C}{u(Q_j)} \sum_{k\geq 1} 2^{-k(c\sigma-\theta)} \int_{2^kQ_j} f v u = \frac{A_j}{u(Q_j)}.
	\end{split}
	\end{equation*}
	
	Now, applying part~\ref{item: prop-cubrimientocritico - item b} of Proposition~\ref{prop-cubrimientocritico}, 
	
	\begin{equation*}
	\begin{split}
	\sum_{j\in\mathbb{N}} uv  \left(\left\{x\in Q_j:\frac{M^{\rho,\sigma}_{\textup{glob}}(fv)(x)}{v(x)}>\frac{t}{2}\right\}\right)
	 \leq & \sum_{j\in\mathbb{N}} uv  \left(\left\{x\in Q_j:\frac{A_j}{u(Q_j)v(x)}>\frac{t}{2}\right\}\right)
	\\ & \leq \frac{C}{t} \sum_{j\in\mathbb{N}} \frac{A_j}{u(Q_j)}\int_{Q_j}  u(x)\,dx 
	\\ & \leq \frac{C}{t} \sum_{j\in\mathbb{N}} \sum_{k\in\mathbb{N}}
	2^{-k(c\sigma-\theta)}  \int_{2^{k}Q_j} f v u
	\\ & \leq \frac{C}{t} \sum_{k\in\mathbb{N}} 2^{-k(c\sigma-\theta)} \int_{\mathbb{R}^d}
	\left(\sum_{j\in\mathbb{N}}\mathcal{X}_{2^{k}Q_j}(y)\right)
	f(y)v(y)u(y)\,dy
	\\ &   \leq \frac{C}{t} \int_{\mathbb{R}^d} f(y) v(y)u(y),
	\end{split}
	\end{equation*}
	choosing $\sigma>(N_1+\theta + 1)/c$.
	
	To deal with the term corresponding to $M^{\rho}_{\textup{loc}}$ we will use the local theory developed in Lemma~\ref{lem-Mloc a Mdiad} and Proposition~\ref{Prop-maxRdiad}. 
	We first claim that given $Q_j$, there exists a cube $R_j\supseteq Q_j$,  $R_j = Q(x_j, C_\rho \rho(x_j))$ such that if $x\in Q_j\cap Q$ and $Q \in \mathcal{Q}_\rho$ then $Q\subseteq R_j$ . To see this, let $x\in Q_j\cap Q$ with $Q = Q(x_Q,\rho(x_Q))$. If $z\in Q$, using~\eqref{eq-constantesRho},
	\begin{equation*}
		\begin{split}
			|x_j-z| & \leq |x_j-x| + |x-z| 
			\\ & \leq \rho(x_j) + 2\rho(x_Q) 
			\\ & \leq \rho(x_j) + 2 C_0 2^{N_0} \rho(x)
			\\ & \leq \rho(x_j) + C_0^2 2^{N_0+2}\rho(x_j).
		\end{split}
	\end{equation*}

	Therefore, we have that
	\begin{equation*}
		\begin{split}
			\sum_{j\in\mathbb{N}} uv\left(\left\{x\in Q_j:\frac{M^\rho_{\textup{loc}}(fv)(x)}{v(x)}>\frac{t}{2}\right\}\right) 
			& \leq 
			\sum_{j\in\mathbb{N}} uv\left(\left\{x\in Q_j:\frac{M_{R_j}(fv)(x)}{v(x)}>\frac{t}{2}\right\}\right)
			\\ & \leq 
			\sum_{j\in\mathbb{N}} uv\left(\left\{x\in R_j:\frac{M_{R_j}(fv)(x)}{v(x)}>\frac{t}{2}\right\}\right)
			\\ & \leq  C
			\sum_{j\in\mathbb{N}}
			\sum_{i =1}^{3^d}uv\left(\left\{x\in P_{i,j}:\frac{M^{\mathscr{D}}_{P_{i,j}}(fv)(x)}{v(x)}>\frac{t}{3^d2}\right\}\right), 
		\end{split}
	\end{equation*}
where in the last inequality we used Lemma~\ref{lem-Mloc a Mdiad}.

We have to make some observations in order to apply Proposition~\ref{Prop-maxRdiad}. First, notice that if $u\in A_1^\rho$ and $v\in A_\infty^\rho(u)$, then there exists $\theta>0$ such that $u\in A_1^{\rho,\theta}$ and $v\in A_\infty^{\rho,\theta}(u)$ with constants $[u]_{A_1^{\rho,\theta}}$, $[v]_{A_\infty^{\rho,\theta}(u)}$. Therefore, $u\in A_1^{\rho,\textup{loc}}$ and $v\in A_\infty^{\rho,\textup{loc}}(u)$ with constants $[u]_{A_1^{\rho,\theta}}$, $[v]_{A_\infty^{\rho,\theta}(u)}$.
Now, by means of Corollary~1 in~\cite{BHS-classesofweights} we have that  $u\in A_1^{\lambda C_\rho \rho,\textup{loc}}$ and $v\in A_\infty^{\lambda C_\rho \rho,\textup{loc}}(u)$. Therefore we have that $u\in A_1(P_{i,j})$ and $v\in A_\infty(P_{i,j},u)$, with constants $C(\rho,d)[u]_{A_1^{\rho,\theta}}$ and $C(\rho,d)[v]_{A_\infty^{\rho,\theta}(u)}$ respectively, since $R_j = C_\rho Q_j$ and $P_{i,j}\subseteq \lambda R_j$, for each $j\in \mathbb{N}$ and $1\leq i\leq 3^d$. The thesis follows by combining Proposition~\ref{Prop-maxRdiad} with part~\ref{item: prop-cubrimientocritico - item b} of Proposition~\ref{prop-cubrimientocritico}. \qedhere
\end{proof}

Before giving a proof for Theorem~\ref{thm: mixta para T SCZO (infinito,delta)} and Theorem~\ref{thm: mixta para T SCZO (s,delta)} we need to establish the following tools. The first one assures that the smoothness estimate for SCZO holds also with an extra decay if necessary. For a proof we refer to Lemma 4 in~\cite{BHQ1}.

\begin{lem}	
Let $\delta>0$ and $T$ be a SCZO of $(s,\delta')$  type for   some $1<s\leq\infty$ and for every $\delta'\in(0,\delta)$. Then, for each $N\in\NN$, there exists a constant $C_N$ such that the associated kernel $K$ satisfies
	\begin{equation*}\label{suav-hormDECAY}
	\left(
	\int_{R<|x- y_0|<2R} |K(x,y)-K(x,y_0)|^{s} dy\right)^{1/s} \\
	\leq C_N R^{-d/s'} \left(\frac{r}{R}\right)^{\delta'}
	\left(1+\frac{R}{\rho(x)}\right)^{-N},
	\end{equation*}
	for $|y-y_0|<r\leq \rho (y_0)$,  $r<R/2$ if $s<\infty$ and
	\begin{equation*}\label{suav-puntualdecay}
	|K(x,y)-K(x,y_0)|
	\leq C_N \frac{|y-y_0|^{\de}}{|x-y|^{d+\de}}\left(1+\frac{|x-y|}{\rho(x)}\right)^{-N},\,\,\,\text{when}\,\,
	|x-y|>2|y-y_0|,
	\end{equation*}
	if $s=\infty$. In particular, this applies to any SCZO of $(s,\delta)$ type.
	
\end{lem}

The next geometrical lemma gives an equivalent size condition for SCZO of $(s,\delta)$ type when $s<\infty$ that will be used to prove Theorem~\ref{thm: mixta para T SCZO (s,delta)}. For a proof we refer to Lemma~3 in~\cite{BHQ2}.

\begin{lem}
Let $1<s<\infty$ and $T$ be a SCZO of $(s,\delta)$ type. Then
for each $N>0$ there exists $C_N$ such that  		\begin{equation*}
		\left(\int_{B(x_0,R/2)}|K(x,y)|^s dx\right)^{1/s}
	 		\leq C_N R^{-d/s'} \left(1+ \frac{R}{\rho(x_0)}\right)^{-N},
	 		\end{equation*}
	 		whenever $R<|y-x_0|<2R$.
\end{lem}

We now proceed with the proof of Theorem~\ref{thm: mixta para T SCZO (infinito,delta)}.

\begin{proof}[Proof of Theorem~\ref{thm: mixta para T SCZO (infinito,delta)}]
	We shall first assume, without loss of generality, that $f$ is nonnegative. Let $\{Q_j\}_{j\in\mathbb{N}}$ be the family of critical cubes given by Proposition~\ref{prop-cubrimientocritico}. We fix $t>0$ and write
	\begin{equation*}
	\begin{split}
	uv\left(\left\{x\in\mathbb{R}^d:\frac{|T(fv)(x)|}{v(x)}>t\right\}\right)
	& \leq \sum_{j\in\mathbb{N}} uv\left(\left\{x\in Q_j:\frac{|T(fv)(x)|}{v(x)}>t\right\}\right)
	\\ & \leq  \sum_{j\in\mathbb{N}} uv\left(\left\{x\in Q_j:\frac{|T(f_1^j v)(x)|}{v(x)}>\frac{t}{2}\right\}\right)
	\\ & +  \sum_{j\in\mathbb{N}} uv\left(\left\{x\in Q_j:\frac{|T(f_2^jv)(x)|}{v(x)}>\frac{t}{2}\right\}\right),
	\end{split}
	\end{equation*}	
	where $f_1^j=f\mathcal{X}_{2Q_j}$ and $f_2^j=f\mathcal{X}_{(2Q_j)^c}$.
	
	We shall first bound the term corresponding to $f_2^j$. By Chebyshev inequality we write 
	\begin{equation*}
	\begin{split}
	\sum_{j\in\mathbb{N}} uv  \left(\left\{x\in Q_j:\frac{|T(f_2^jv)(x)|}{v(x)}>\frac{t}{2}\right\}\right)
 & \leq   \sum_{j\in\mathbb{N}}\frac{2}{t} \int_{Q_j} |T(f_2^jv)|(x)u(x)\,dx
	\\& \leq   \sum_{j\in\mathbb{N}}\frac{2}{t}\int_{Q_j} \left(\sum_{k\in\mathbb{N}}\int_{2^{k+1}Q_j\setminus 2^kQ_j}
	|K(x,y)|f(y)v(y)\,dy \right)u(x)\,dx.
	\end{split}
	\end{equation*}
	Now, by applying \eqref{TamPuntual} and $A_1^\rho$ condition for $u$, we have that 
	
	\begin{equation*}
	\begin{split}
	\sum_{j\in\mathbb{N}} uv  \left(\left\{x\in Q_j:\frac{|T(f_2^jv)(x)|}{v(x)}>\frac{t}{2}\right\}\right)
 & \leq \frac{C_N}{t} \sum_{j\in\mathbb{N}}\int_{Q_j}\sum_{k\in\mathbb{N}}
	\frac{2^{-kN}}{(2^k\rho(x_j))^d}\left(\int_{2^{k+1}Q_j}
	f(y) v(y)\,dy \right) u(x)\,dx 
	\\ & \leq \frac{C_N}{t} \sum_{j\in\mathbb{N}} \sum_{k\in\mathbb{N}}
	2^{-kN}  \frac{u(2^{k+1}Q_j)}{|2^{k+1}Q_j|}\int_{2^{k+1}Q_j}
	f(y) v(y)\,dy 
	\\ & \leq \frac{C_N}{t} \sum_{j\in\mathbb{N}} \sum_{k\in\mathbb{N}}
	2^{-k(N-\theta)}  \int_{2^{k+1}Q_j}
	f(y) v(y)u(y)\,dy
	\\ & \leq \frac{C_N}{t} \sum_{k\in\mathbb{N}} 2^{-k(N-\theta)} \int_{\mathbb{R}^d}
	\left(\sum_{j\in\mathbb{N}}\mathcal{X}_{2^{k+1}Q_j}(y)\right)
	f(y)v(y)u(y)\,dy
	\\ &   \leq \frac{C}{t} \int_{\mathbb{R}^d} f(y) v(y)u(y),
	\end{split}
	\end{equation*}
	where we have used part~\ref{item: prop-cubrimientocritico - item b} of Proposition~\ref{prop-cubrimientocritico} and chosen $N>N_1+\theta + 1$.

	We now turn our attention to the term corresponding to $f_1^j$. We shall prove that 
	\begin{equation}\label{eq: thm: mixta para T SCZO (infinito,delta) - eq1}
	 uv\left(\left\{x\in Q_j:\frac{|T(f_1^j v)(x)|}{v(x)}>\frac{t}{2}\right\}\right)\leq \frac{C}{t}\int_{2Q_j} f(x)v(x)u(x)\,dx,
	\end{equation}
	for each $j\in\mathbb{N}$. Then the result directly follows by applying part \ref{item: prop-cubrimientocritico - item b} of Proposition~\ref{prop-cubrimientocritico}. Fix $j\in\mathbb{N}$ and let us first suppose that \refstepcounter{BPQ}\label{pag: estimacion de f_1^j, promedio mayor que t/2}
	\[(f_1^j)_{2Q_j}^v=\frac{1}{v(2Q_j)}\int_{2Q_j}f_1^j(x)v(x)\,dx\geq \frac{t}{2}.\]
	Then,  by Corollary~\ref{cor: lema fundamental}
	\begin{equation*}
	\begin{split}
	 uv\left(\left\{x\in Q_j:\frac{|T(f_1^j v)(x)|}{v(x)}>\frac{t}{2}\right\}\right) & 	
 \leq \frac{uv(2Q_j)v(2Q_j)}{v(2Q_j)} \\
 & \leq \frac{2uv(2Q_j)}{v(2Q_j)t}\int_{2Q_j}f_1^j(x)v(x)\,dx
	\\ &\leq \frac{C}{t}\left(\inf_{2Q_j}u\right)\int_{2Q_j}f(x)v(x)\,dx
	\\ &\leq \frac{C}{t} \int_{2Q_j}f(x)v(x)u(x)\,dx.
	\end{split}
	\end{equation*}
	
	On the other hand, if $(f_1^j)_{2Q_j}^v<t/2$, we apply the decomposition given by Lemma~\ref{lema: DCZ en un cubo fijo} on $2Q_j$ at level $t/2$, obtaining a collection of pairwise disjoint dyadic cubes $\{P_{i,j}\}_{i\in\mathbb{N}}$ in $\mathscr{D}(2Q_j)$ and a pair of functions $g_j$ and $h_j$ such that $f_1^j=g_j+h_j$. By setting  $\tilde{P}_{i,j}=3P_{i,j}$ and $\tilde{\Omega}_j=\bigcup_{i\in\mathbb{N}} \tilde{P}_{i,j}$, we write
	
	\begin{equation*}
	\begin{split}
		 uv\left(\left\{x\in Q_j:\frac{|T(f_1^j v)(x)|}{v(x)}>\frac{t}{2}\right\}\right)  	
		 &\leq  uv\left(\left\{x\in Q_j:\frac{|T(g_j v)(x)|}{v(x)}>\frac{t}{4}\right\}\right)
		\\ & \qquad  + uv\left(\left\{x\in Q_j:\frac{|T(h_j v)(x)|}{v(x)}>\frac{t}{4}\right\}\right)
		\\ &\leq  uv\left(\left\{x\in Q_j:\frac{|T(g_j v)(x)|}{v(x)}>\frac{t}{4}\right\}\right)
	 + uv(\tilde{\Omega}_j)
		\\ & \qquad  + uv\left(\left\{x\in Q_j\setminus \tilde{\Omega}_j:\frac{|T(h_j v)(x)|}{v(x)}>\frac{t}{4}\right\}\right)
		\\ & = I + II + III.
	\end{split}
\end{equation*}
	Let us first estimate $I$. Since $v\in A_\infty^{\rho}(u)$, there exists $1<q'<\infty$ such that $v\in A_{q'}^\rho(u)$. Therefore, $v^{1-q}\in A_q^\rho(u)$ by item~\ref{item: lema: propiedades clase Ap,rho - item a} of Lemma~\ref{lema: propiedades clase Ap,rho}. Consequently, we have that $uv^{1-q}\in A_q^\rho$ by virtue of Lemma~\ref{lema: u en A1,rho y v en Ap,rho implican uv en Ap,rho}. Then, by applying Chebyshev inequality with exponent $q$ and using the strong $(q,q)$ type of $T$ for $A_q^\rho$ weights (see Remark~\ref{rem: tipo fuerte SCZ infinito}) we obtain  
	\begin{align*}
	I = uv\left(\left\{x\in Q_j:\frac{|T(g_j v)(x)|}{v(x)}>\frac{t}{4}\right\}\right)&\leq \frac{C}{t^q} \int_{\mathbb{R}^d} |T(g_j v)(x)|^q u(x)v^{1-q}(x)\,dx\\
	 & \leq \frac{C}{t^q} \int_{2Q_j} |g_j(x)v(x)|^q u(x)v^{1-q}(x)\,dx\\
	 & \leq \frac{C}{t} \int_{2Q_j} g_j(x) u(x)v(x)\,dx,
	\end{align*}
	where we have also used that $g_j(x)\leq Ct$ in almost every $x\in2Q_j$. Finally, from the definition of $g_j$ and Corollary~\ref{cor: lema fundamental} we arrive to \refstepcounter{BPQ}\label{pag: estimacion de I}
	\begin{align*}
	I & \leq \frac{C}{t} \int_{2Q_j\setminus \Omega_j} f(x) u(x)v(x)\,dx + \frac{C}{t} \sum_{i\in\mathbb{N}} \frac{uv(P_{i,j})}{v(P_{i,j})} \int_{P_{i,j}} f(x)v(x)\,dx\\
	 & \leq  \frac{C}{t} \int_{2Q_j\setminus \Omega_j} f(x) u(x)v(x)\,dx + \frac{C}{t} \sum_{i\in\mathbb{N}}  \int_{P_{i,j}} f(x)u(x)v(x)\,dx\\
	 & \leq  \frac{C}{t} \int_{2Q_j} f(x) u(x)v(x)\,dx. 
	\end{align*}	
	
	To deal with $II$, observe that $uv\in A_\infty^\rho$ by Lemma~\ref{lema: u en A1,rho y v en Ap,rho implican uv en Ap,rho} and therefore it is a doubling weight over sub-critical cubes. We proceed as follows  \refstepcounter{BPQ}\label{pag: estimacion de II}
	\begin{equation*}
	\begin{split}
	II  = uv(\tilde{\Omega}_j) \leq \sum_{i\in\mathbb{N}} uv(\tilde{P}_{i,j})
	 & \leq  C \sum_{i\in\mathbb{N}} v(P_{i,j}) \frac{uv(P_{i,j})}{v(P_{i,j})}
	\\ & \leq  \frac{C}{t} \sum_{i\in\mathbb{N}} \left(\inf_{P_{i,j}}u\right)  \int_{P_{i,j}} f(x)v(x)\,dx
	\\  & \leq  \frac{C}{t} \int_{2Q_j} f(x)u(x)v(x)\,dx.
	\end{split}
	\end{equation*}
	
	In order to estimate $III$, by using item~\ref{item: lema: DCZ en un cubo fijo - item c} of Lemma~\ref{lema: DCZ en un cubo fijo} we can write \refstepcounter{BPQ}\label{pag: estimacion de III}
	\begin{equation*}
	\begin{split}
	 III  & = uv\left(\left\{x\in Q_j\setminus \tilde{\Omega}_j:\frac{|T(h_jv)(x)|}{v(x)}>\frac{t}{4}\right\}\right)
	\\ & \leq \frac{C}{t} \sum_{i\in\mathbb{N}} \int_{2Q_j\setminus \tilde{P}_{i,j}} |T(h_{i,j}v)(x)|u(x)\,dx
	\\ & \leq \frac{C}{t} \sum_{i\in\mathbb{N}} \int_{2Q_j\setminus \tilde{P}_{i,j}} \left|\int_{P_{i,j}} K(x,y)h_{i,j}(y)v(y)\,dy\right|u(x)\,dx 
	\\ & \leq \frac{C}{t} \sum_{i\in\mathbb{N}} \int_{2Q_j\setminus \tilde{P}_{i,j}} \left|\int_{P_{i,j}} \left[K(x,y)-K(x,x_{i,j})\right]h_{i,j}(y)v(y)\,dy\right|u(x)\,dx 
	\\ & \leq \frac{C}{t} \sum_{i\in\mathbb{N}} \int_{P_{i,j}} |h_{i,j}(y)|v(y)   \int_{2Q_j\setminus \tilde{P}_{i,j}} \left|K(x,y)-K(x,x_{i,j})\right|u(x)\,dx\,dy. 
	\end{split}
	\end{equation*}
	
	Let us denote $r_{i,j}$ the radius of the ball inscribed in $\tilde{P}_{i,j}$ and let $\theta\geq0$ such that $u\in A_1^{\rho,\theta}$. For $y\in P_{i,j}$, by applying condition \eqref{suav-puntual}  we have that
	
	\begin{equation*}
	\begin{split}
  \int_{2Q_j\setminus \tilde{P}_{i,j}} &\left|K(x,y)-K(x,x_{i,j})\right|u(x)\,dx
  \\ & \leq C_N \sum_{k\in\mathbb{N}} \int_{2^{k-1}r_{i,j}\leq |x-x_{i,j}|<2^k r_{i,j}} \frac{|y-x_{i,j}|^{\delta}}{|x-x_{i,j}|^{d+\delta}} \left(1+\frac{|x-x_{i,j}|}{\rho(x_{i,j})}\right)^{-N}u(x)\,dx
   \\ & \leq C_N \sum_{k\in\mathbb{N}}
   \frac{r_{i,j}^{\delta}}{(2^k r_{i,j})^{d+\delta}}
   \left(1+\frac{2^k r_{i,j}}{\rho(x_{i,j})}\right)^{-N}
    \int_{|x-x_{i,j}|<2^k r_{i,j}}  u(x)\,dx
     \\ & \leq C_N u(y) \sum_{k\in\mathbb{N}}
    2^{-k\delta}
    \left(1+\frac{2^k r_{i,j}}{\rho(x_{i,j})}\right)^{-N+\theta} \leq C u(y),
	\end{split}
	\end{equation*}
	where we also used that $u\in A_1^\rho$ and chose $N>\theta$. Therefore, by Corollary~\ref{cor: lema fundamental} we get \refstepcounter{BPQ}\label{pag: estimacion de III-2}
	
	\begin{equation*}
	\begin{split}
	III  &  \leq \frac{C}{t} \sum_{i\in\mathbb{N}} \int_{P_{i,j}} |h_{i,j}(y)|u(y)v(y)\,dy 
	\\ &  \leq \frac{C}{t} \sum_{i\in\mathbb{N}} \left(
	\int_{P_{i,j}} f_1^{j}(y)u(y)v(y)\,dy  
	+\int_{P_{i,j}} (f_1^j)^v_{P_{i,j}}u(y)v(y)\,dy \right)
	\\ &  \leq \frac{C}{t} \sum_{i\in\mathbb{N}} \left(
	\int_{P_{i,j}} f_1^{j}(y)u(y)v(y)\,dy  
	+ \frac{uv(P_{i,j})}{v(P_{i,j})}\int_{P_{i,j}} f(y)v(y)\,dy \right)
	\\ &  \leq \frac{C}{t} \int_{2Q_j} f(y)u(y)v(y)\,dy.
	\end{split}
	\end{equation*}
	This completes the proof of \eqref{eq: thm: mixta para T SCZO (infinito,delta) - eq1}. By summing over $j$ we can conclude the thesis. \qedhere
\end{proof}


\begin{proof}[Proof of Theorem~\ref{thm: mixta para T SCZO (s,delta)}]
	The proof will follow similar lines as the previous one, so we will just focus on the parts where suitable changes are needed. We consider again	 $\{Q_j\}_{j\in\mathbb{N}}$ the family of critical cubes given by Proposition~\ref{prop-cubrimientocritico}. Fix $t>0$ and  write
	\begin{equation*}
	\begin{split}
	uv\left(\left\{x\in\mathbb{R}^d:\frac{|T(fv)(x)|}{v(x)}>t\right\}\right)
	& \leq \sum_{j\in\mathbb{N}} uv\left(\left\{x\in Q_j:\frac{|T(fv)(x)|}{v(x)}>t\right\}\right)
	\\ & \leq  \sum_{j\in\mathbb{N}} uv\left(\left\{x\in Q_j:\frac{|T(f_1^j v)(x)|}{v(x)}>\frac{t}{2}\right\}\right)
	\\ & \qquad+  \sum_{j\in\mathbb{N}} uv\left(\left\{x\in Q_j:\frac{|T(f_2^jv)(x)|}{v(x)}>\frac{t}{2}\right\}\right),
	\end{split}
	\end{equation*}	
	where $f_1^j=f\mathcal{X}_{2Q_j}$ and $f_2^j=f\mathcal{X}_{(2Q_j)^c}$.
	
	We will first bound the term corresponding to $f_2^j$. By applying Chebyshev inequality we have 
	\begin{equation*}
	\begin{split}
	\sum_{j\in\mathbb{N}} uv & \left(\left\{x\in Q_j:\frac{|T(f_2^jv)(x)|}{v(x)}>\frac{t}{2}\right\}\right)
	\\ & \leq   \sum_{j\in\mathbb{N}}\frac{2}{t} \int_{Q_j} |T(f_2^jv)(x)|u(x)\,dx
	\\& \leq   \sum_{j\in\mathbb{N}}\frac{2}{t}\int_{Q_j} \left(\int_{(2Q_j)^c}
	|K(x,y)|f(y)v(y)\,dy \right) u(x)\,dx
	\\& \leq   \sum_{j\in\mathbb{N}}\frac{2}{t}\int_{(2Q_j)^c}  f(y)v(y) \left(\int_{Q_j}
	|K(x,y)| u(x)\,dx\right)\,dy
	\\& \leq   \sum_{j\in\mathbb{N}}\frac{2}{t} \sum_{k\in\mathbb{N}}\int_{2^{k+1}Q_j\setminus 2^kQ_j}  f(y)v(y) \left(\int_{Q_j}
	|K(x,y)|^{s}\,dx\right)^{\frac{1}{s}}\left(\int_{Q_j}
	 u^{s'}(x)\,dx\right)^{\frac{1}{s'}}\,dy.
	\end{split}
	\end{equation*}
	
	By applying condition~\eqref{TamHorm} on $K$, for each $N$ we have that 
	\begin{equation*}
	\begin{split}
	\sum_{j\in\mathbb{N}} uv & \left(\left\{x\in Q_j:\frac{|T(f_2^jv)(x)|}{v(x)}>\frac{t}{2}\right\}\right)
	\\ & \leq \frac{C_N}{t} \sum_{j\in\mathbb{N}} \sum_{k\in\mathbb{N}} |2^kQ_j|^{-1/s'} 2^{-kN}
	\int_{2^{k+1}Q_j}f(y)v(y) 
	\left(\int_{Q_j}u^{s'}\right)^{1/s'}\,dy
	\\ & \leq \frac{C_N}{t} \sum_{j\in\mathbb{N}} \sum_{k\in\mathbb{N}}  2^{-kN}
	\int_{2^{k+1}Q_j}f(y)v(y) 
	\left(\frac{1}{|2^{k+1}Q_j|}\int_{2^{k+1}Q_j}u^{s'}\right)^{1/s'}\,dy
	\\ & \leq \frac{C_N}{t} \sum_{j\in\mathbb{N}} \sum_{k\in\mathbb{N}}  2^{-k(N-\theta/s')}
	\int_{2^{k+1}Q_j}f(y)v(y) 
	u(y)\,dy
	\\ & \leq \frac{C_N}{t} \sum_{k\in\mathbb{N}} 2^{-k(N-\theta/s')} \int_{\mathbb{R}^d}
	\left(\sum_{j\in\mathbb{N}}\mathcal{X}_{2^{k+1}Q_j}(y)\right)
	f(y)v(y)u(y)\,dy
	\\ &   \leq \frac{C}{t} \int_{\mathbb{R}^d} f(y) v(y)u(y),
	\end{split}
	\end{equation*}
	where we have used part~\ref{item: prop-cubrimientocritico - item b} of Proposition~\ref{prop-cubrimientocritico} and chosen $N>N_1+\theta/s' + 1$.

	For $f_1^j$ we shall prove that, for every $j\in\mathbb{N}$, the inequality 
	\begin{equation}\label{eq: thm: mixta para T SCZO (s,delta) - eq1}
	uv\left(\left\{x\in Q_j:\frac{|T(f_1^j v)(x)|}{v(x)}>\frac{t}{2}\right\}\right)\leq \frac{C}{t}\int_{2Q_j} f(x)v(x)u(x)\,dx
	\end{equation}
	holds.
	Fixed $j\in\mathbb{N}$,  we can assume that $(f_1^j)_{2Q_j}^v<t/2$, since the estimate follows exactly as in page~\pageref{pag: estimacion de f_1^j, promedio mayor que t/2} otherwise.  We apply the decomposition given by Lemma~\ref{lema: DCZ en un cubo fijo} on $2Q_j$ at level $t/2$ obtaining the family $\{P_{i,j}\}_{i\in\mathbb{N}}$, $g_j$ and $h_j$ such that $f_1^j=g_j+h_j$. By adopting the same notation as in the proof of Theorem~\ref{thm: mixta para T SCZO (infinito,delta)} we get
	\begin{equation*}
	\begin{split}
	uv\left(\left\{x\in Q_j:\frac{|T(f_1^j v)(x)|}{v(x)}>\frac{t}{2}\right\}\right)  	
	 &\leq  uv\left(\left\{x\in Q_j:\frac{|T(g_j v)(x)|}{v(x)}>\frac{t}{4}\right\}\right)+uv(\tilde{\Omega_j})
	\\ & \qquad  + uv\left(\left\{x\in Q_j\setminus \tilde{\Omega_j}:\frac{|T(h_j v)(x)|}{v(x)}>\frac{t}{4}\right\}\right)
	\\ & = I + II + III.
	\end{split}
	\end{equation*}
	We shall first estimate $I$. By Proposition~\ref{prop: peso para el tipo fuerte de T (s,delta)}, our hypotheses on $u$ and $v$ imply that there exists a number $1<q<s$ such that $w^{1-q'}=u^{1-q'}v\in A_{q'/s'}^\rho$. By applying Chebyshev inequality with exponent $q$ we obtain that 
	\begin{equation*}
	\begin{split}
	I & \leq \frac{C}{t^q} \int_{Q_j} |T(g_j v)(x)|^q u(x)v^{1-q}(x)\,dx
	\\ & \leq \frac{C}{t^q} \int_{Q_j} |T(g_j v)(x)|^q w(x)\,dx
	\\ & \leq \frac{C}{t^q} \int_{2Q_j} g_j(x)^q u(x)v(x)\,dx
	\\ & \leq \frac{C}{t} \int_{2Q_j} g_j(x) u(x)v(x)\,dx,
	\end{split}
	\end{equation*}
	where we have applied the boundedness of $T$ on $L^q(w)$ stated in Remark~\ref{rem: tipo fuerte SCZ s} and used that $g_j(x)\leq Ct$ for almost every $x\in 2Q_j$. From this point we can obtain the desired estimate for $I$ by repeating the argument given in page~\pageref{pag: estimacion de I}. 

		The estimate of $II$ can be performed exactly as we have done in page~\pageref{pag: estimacion de II}.	
%
	It only remains to estimate $III$. By proceeding as in page~\pageref{pag: estimacion de III} we have that
	\[III\leq \frac{C}{t} \sum_{i\in\mathbb{N}} \int_{P_{i,j}} |h_{i,j}(y)|v(y)   \int_{2Q_j\setminus \tilde{P_{i,j}}} \left|K(x,y)-K(x,x_{i,j})\right|u(x)\,dx\,dy.\]	
	Let $r_{i,j}$ be as in the proof of Theorem~\ref{thm: mixta para T SCZO (infinito,delta)} and $\theta\geq 0$ such that $u^{s'}\in A_1^{\rho,\theta}$. By applying condition \eqref{suav-horm} on the kernel, for $y\in P_{i,j}$ we obtain that
	
	\begin{equation*}
	\begin{split}
	\int_{2Q_j\setminus\tilde{P}_{i,j}} &\left|K(x,y)-K(x,x_{i,j})\right|u(x)\,dx
	\\ & \leq \sum_{k\in\mathbb{N}} \int_{ 2^{k-1}r_{i,j}\leq |x-x_{i,j}|<2^k r_{i,j}} 
	\left|K(x,y)-K(x,x_{i,j})\right|u(x)\,dx
	\\ & \leq \sum_{k\in\mathbb{N}} \left(\int_{ 2^{k-1}r_{i,j}\leq |x-x_{i,j}|<2^k r_{i,j}}
	\left|K(x,y)-K(x,x_{i,j})\right|^s\,dx\right)^{1/s}
	\left(\int_{ |x-x_{i,j}|<2^k r_{i,j}}  u^{s'}(x)\,dx\right)^{1/s'}
	\\ & \leq C_N \sum_{k\in\mathbb{N}} (2^k r_{i,j})^{-d/s'} 2^{-k\delta} 
	\left(1+\frac{2^kr_{i,j}}{\rho(x_{i,j})}\right)^{-N}
	(2^kr_{i,j})^{d/s'} \left(1+\frac{2^kr_{i,j}}{\rho(x_{i,j})}\right)^{\theta/s'}
	u(y)
	\\ & \leq C_N u(y) \sum_{k\in\mathbb{N}}
	2^{-k\delta}
	\left(1+\frac{2^k r_{i,j}}{\rho(x_{i,j})}\right)^{-N+\theta/s'}
	\\ & \leq C u(y),
	\end{split}
	\end{equation*}
	where we have chosen $N>\theta/s'$. Therefore
	\[III \leq \frac{C}{t} \sum_{i\in\mathbb{N}} \int_{P_{i,j}} |h_{i,j}(y)|u(y)v(y)\,dy\]
	and we can conclude the estimate as in page~\pageref{pag: estimacion de III-2}. This completes the proof of \eqref{eq: thm: mixta para T SCZO (s,delta) - eq1} and we are done.
\end{proof}

%
%
%

\section{Application: Schr\"odinger type singular integrals}\label{seccion: aplicaciones}

In this section we are going to apply the results obtained along this work to some operators associated to the Schr\"odinger semigroup generated by $L= -\Delta + V$ on $\mathbb{R}^d$,  $d\geq 3$. We will assume that   the potential $V$ is a non-negative function, not identically zero and satisfying a reverse-H\"older condition of order $q>d/2$, this is 
\[
\left( \frac{1}{|B|} \int_B V^q\right) ^{1/q} \leq C \frac{1}{|B|} \int_B V,
\]
for all ball $B$. As usual, we denote this by $V\in \textup{RH}_q$. 

In~\cite{shen}, Shen defined the function \[
\rho(x)= \sup\left\lbrace r>0: \frac{1}{r^{d-2}} \int_{B(x,r)}V(x)\,dx \leq 1 \right\rbrace.
\]
and proved that is a critical radius function, under the above assumptions. There, he also proved the boundedness on Lebesgue spaces of some singular integral operators associated to $L$. Later, Shen's results were extended in several directions, considering different function spaces or different type of inequalities concerning these singular integral operators. 

Our purpose in this section is to obtain mixed weak type inequalities for the Schr\"odinger Riesz transforms of first and second order $\mathcal{R}_1=\nabla L^{-1/2}$ and $\mathcal{R}_2=\nabla^2 L^{-1}$ as well as for the operators $T_\gamma= V^\gamma L^{-\gamma}$ for $0<\gamma<d/2$ and $S_\gamma = V^{\gamma-1/2}\nabla L^{-\gamma}$ for $1/2\leq \gamma<1 $.

As it is implied by Shen's work, if $V\in \textup{RH}_q$, with $q<\infty$ then the only candidate to be a SCZO of $(\infty,\delta)$ type is $\mathcal{R}_1$ provided $q>d$. The rest of the operators above are not even bounded in the whole range  $1<p<\infty$.

We summarize in the next proposition some known properties for the above operators. A proof can be found in~\cite{BHQ1} where SCZO classes were widely discussed, however most of the required estimates may be traced back to~\cite{shen}.

\begin{prop}\label{prop-lista}
	Let $V\geq 0$ satisfying a reverse H\"older inequality of order $q>d/2$ with $d\geq3$.
	Then, for some $\delta>0$, which may be different at each occurrence, we have:
	\begin{enumerate}[\rm(a)]
		\item\label{item-lista1}
		$ \mathcal{R}_1$ is a SCZO of $(\infty,\delta)$ type if we further ask $q\geq d$.
		\item\label{item-lista2}
		$\mathcal{R}_1$ is a SCZO of $(p_0,\delta)$ type, with $p_0$ such that $1/p_0=1/q-1/d$.
		\item\label{item-lista3}
		$\mathcal{R}_2$ is a SCZO of $(q,\delta)$ type.
		\item\label{item-lista4}
		$T_\gamma$ is a SCZO of $(q/\gamma,\delta)$  type for $0<\gamma<d/2$.
		\item\label{item-lista5}
		$S_\gamma$ is a SCZO of $(q_\gamma,\delta)$ type  where $q_\gamma$ is  such that 
			\begin{equation*}\label{eq-qgamma}
		1/q_\gamma=\left( 1/q -1/d\right)^+ +(2\gamma-1) /2q
		\end{equation*}
		 with $1/2<\gamma\leq 1$.
	\end{enumerate}
\end{prop}

The proposition above, together with Theorem~\ref{thm: mixta para T SCZO (infinito,delta)} and Theorem~\ref{thm: mixta para T SCZO (s,delta)} allow us to obtain mixed inequalities for the singular integrals operators associated to $L$. We state the results in the following theorem. When considering $\mathcal{R}_1$, the openness of the reverse H\"older condition allow us to consider the cases $d/2<q<d$ and $q>d$.

\begin{thm}
	Let $V\in \textup{RH}_q$ for $q>d/2$. Then the following inequalities hold for every positive $t$.
	\begin{enumerate}[\rm(a)]
	    \item 
	 If $d<q<\infty$, $u\in A_1^\rho$ and $v\in A_\infty^\rho(u)$, then
	\[uv\left(\left\{x\in\mathbb{R}^d: \frac{|\mathcal{R}_1(fv)(x)|}{v(x)}>t\right\}\right)\leq \frac{C}{t}\int_{\mathbb{R}^d}|f(x)|u(x)v(x)dx.\]
	
	\item
	If $d/2<q<d$, $u^{p_0'}\in A_1^\rho$ and $v\in A_\infty^\rho(u^\beta)$ for some $\beta>p_0'$ then
	\[uv\left(\left\{x\in\mathbb{R}^d: \frac{|\mathcal{R}_1(fv)(x)|}{v(x)}>t\right\}\right)\leq \frac{C}{t}\int_{\mathbb{R}^d}|f(x)|u(x)v(x)dx.\]

\item
	If $u^{q'}\in A_1^\rho$ and $v\in A_\infty^\rho(u^\beta)$ for some $\beta>q'$ then
	\[uv\left(\left\{x\in\mathbb{R}^d: \frac{|\mathcal{R}_2(fv)(x)|}{v(x)}>t\right\}\right)\leq \frac{C}{t}\int_{\mathbb{R}^d}|f(x)|u(x)v(x)dx.\]

\item
	If $u^{(q/\gamma)'}\in A_1^\rho$ and $v\in A_\infty^\rho(u^\beta)$ for some $\beta>(q/\gamma)'$ then
	\[uv\left(\left\{x\in\mathbb{R}^d: \frac{|T_\gamma(fv)(x)|}{v(x)}>t\right\}\right)\leq \frac{C}{t}\int_{\mathbb{R}^d}|f(x)|u(x)v(x)dx.\]

\item
	If $u^{q_\gamma'}\in A_1^\rho$ and $v\in A_\infty^\rho(u^\beta)$ for some $\beta>q_\gamma'$ then
	\[uv\left(\left\{x\in\mathbb{R}^d: \frac{|S_\gamma(fv)(x)|}{v(x)}>t\right\}\right)\leq \frac{C}{t}\int_{\mathbb{R}^d}|f(x)|u(x)v(x)dx.\]

\end{enumerate}
\end{thm}

As expected, in all of the above theorems we generalize the weighted weak $(1,1)$ type proved in~\cite{BCH3} for $\mathcal{R}_1$, and in~\cite{BHQ1} for $\mathcal{R}_2$, $T_\gamma$ and $S_\gamma$.

\bibliographystyle{amsplain}

\end{document}